%% file: SISC.tex
\documentclass[hidelinks,onefignum,onetabnum,final]{siamart220329}
\input{SISC_shared}

\usepackage{xr}
\makeatletter

\newcommand*{\addFileDependency}[1]{
\typeout{(#1)}
\@addtofilelist{#1}
\IfFileExists{#1}{}{\typeout{No file #1.}}
}\makeatother

\newcommand*{\myexternaldocument}[1]{%
\externaldocument{#1}%
\addFileDependency{#1.tex}%
\addFileDependency{#1.aux}%
}

\myexternaldocument{SISC_supplement}

\date{}

\begin{document}

\maketitle

\begin{abstract}
We analyze randomized matrix-free quadrature algorithms for spectrum and spectral sum approximation. The algorithms studied include the kernel polynomial method and stochastic Lanczos quadrature, two widely used methods for these tasks.  Our analysis of spectrum approximation unifies and simplifies several one-off analyses for these algorithms which have appeared over the past decade. In addition, we derive bounds for spectral sum approximation which guarantee that, with high probability, the algorithms are simultaneously accurate on all bounded analytic functions. Finally, we provide comprehensive and complimentary numerical examples. These examples illustrate some of the qualitative similarities and differences between the algorithms, as well as relative drawbacks and benefits to their use on different types of problems.
\end{abstract}

\section{Introduction}

In this paper, we analyze a general class of randomized quadrature algorithms for approximating the cumulative empirical spectral measure (CESM)\footnote{In physics, the ``density'' \( \d\Phi/\d{x} \) is often called the density of states (DOS).} 
\begin{equation}
    \Phi(x) 
    = \sum_{i=1}^{n} \frac{1}{n} \bOne(\lambda_i\leq x)
\end{equation}
corresponding to a Hermitian matrix $\vec{A}$ with eigendecomposition \( \vec{A} = \sum_{i=1}^{n} \lambda_i \vec{u}_i \vec{u}_i^\cT \).
The CESM is related to the spectral sum $\tr(\fA)$ in that
\begin{equation}
    \tr(\fA) = \sum_{i=1}^{n} f(\lambda_i) = n \int f(x) \d\Phi(x),
\end{equation}
where $\fA = \sum_{i=1}^{n} f(\lambda_i) \vec{u}_i \vec{u}_i^\cT$ is a matrix function.
The algorithms we study include the well-known Kernel Polynomial Method (KPM) \cite{skilling_89,weisse_wellein_alvermann_fehske_06} and Stochastic Lanczos Quadrature (SLQ) \cite{bai_fahey_golub_96,ubaru_chen_saad_17} and can be broken into two main stages, illustrated in \cref{fig:intro_spec_approx}.

\begin{figure}
    \includegraphics[width=\textwidth]{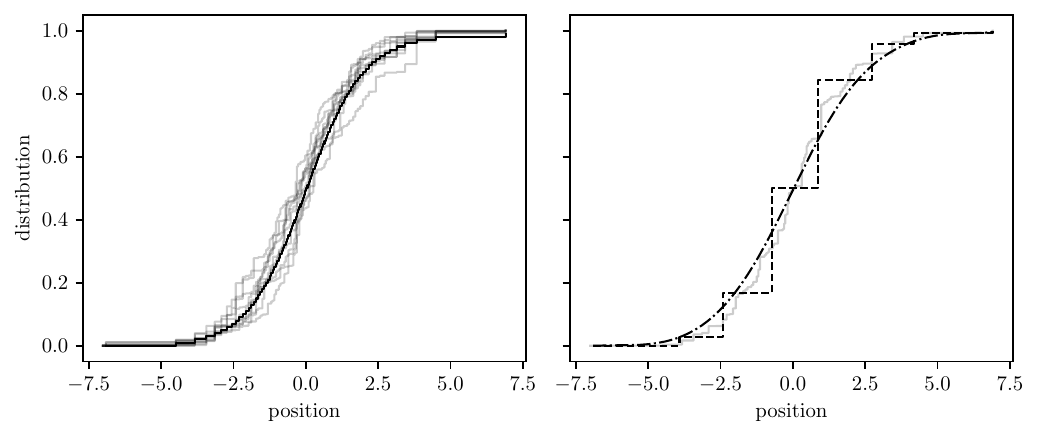}
    \caption{
        Illustration of the two main components of \cref{alg:protoalg}.
        \emph{Left}: 
        CESM \( \Phi \) ({\protect\raisebox{0mm}{\protect\includegraphics[scale=.7]{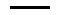}}}) and 
        10 independent samples of weighted CESM \( \Psi \) ({\protect\raisebox{0mm}{\protect\includegraphics[scale=.7]{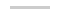}}}).
        Each copy is a sample from an unbiased estimator for \( \Phi \) at all points \( x \).
    \emph{Right}:
        One sample of the weighted CESM \( \Psi \) ({\protect\raisebox{0mm}{\protect\includegraphics[scale=.7]{imgs/legend/solid_thin_alpha.pdf}}}) and different approximations \( \qq{\Psi}{s} \) to \( \Psi \) based on stochastic Lanczos quadrature
        ({\protect\raisebox{0mm}{\protect\includegraphics[scale=.7]{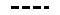}}})
        and the damped kernel polynomial method ({\protect\raisebox{0mm}{\protect\includegraphics[scale=.7]{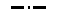}}}).
        Note that while the approximations are both induced by \( \ff{\:\cdot\:}{s} \) for different choices of ``\( \circ \)'' they are qualitatively different: one produces a piecewise constant approximation while the other produces a continuous approximation.
        In different situations, one type of approximation may be preferable to the other.}
    \label{fig:intro_spec_approx}
\end{figure}

In the first stage, the CESM $\Phi$ is approximated with the weighted CESM
\begin{equation}
    \Psi(x) = \vec{v}^\cT \bOne(\vec{A}\leq x) \vec{v}
    = \sum_{i=1}^{n} |\vec{v}^\cT \vec{u}_i|^2 \bOne(\lambda_i\leq x),
\end{equation}
corresponding to $\vec{A}$ and random unit vector $\vec{v}$ satisfying $\EE[\vec{v}\vec{v}^\cT] = n^{-1} \vec{I}$.
Since $\EE[|\vec{v}^\cT \vec{u}_i|^2] = \EE[\tr(\vec{v}\vec{v}^\cT \vec{u}_i \vec{u}_i^\cT)] = n^{-1} \tr(\vec{u}_i^\cT \vec{u}_i) = n^{-1}$, the weighted CESM $\Psi(x)$ is an unbiased estimator for the CESM $\Phi(x)$.
To reduce the variance of the estimator, we use the averaged weighted CESM 
\begin{equation}
    \samp{ \Psi_\ell } = \frac{1}{\nv} \sum_{\ell=1}^{\nv} \Psi_\ell ,
\end{equation}
where \( \{ \Psi_\ell \}_{\ell=1}^{\nv} \) be independent and identically distributed (iid) copies of the weighted CESM \( \Psi \) corresponding to vectors \( \{ \vec{v}_\ell \}_{\ell=1}^{\nv} \) which are iid copies of \( \vec{v} \).
Clearly $\samp{ \Psi_\ell }$ is also an unbiased estimator for the CESM at every point \( x \in \R \).
Note that we are abusing notation slightly, as $\samp{\Psi_\ell}$ does not depend on the index $\ell$; i.e. $\ell$ is a dummy variable.\footnote{This is not unlike how we often write $f(x)$ to mean the function $f:\R\to\R$.}

In the second stage, each weighted CESM $\Psi_\ell$ is approximated using a polynomial quadrature rule. 
In particular, given a polynomial operator \( \ff{\:\cdot\:}{s} \) such that \( \ff{f}{s} \) is a \emph{polynomial} approximation to $f$ of degree at most \( s \), a natural approach to obtaining a \emph{quadrature} approximation \( \qq{\Psi}{s} \) to \( \Psi \) is through the definition\footnote{Provided $\ff[]{f}{s}$ is linear, as is the case for the choices of $\circ$ we consider, the Riesz–Markov–Kakutani representation theorem allows us to define a distribution function in this way.}
\begin{equation}
    \int_{-\infty}^{\infty} f(x) \d \qq{\Psi}{s}(x)  = \int_{-\infty}^{\infty} \ff{f}{s}(x) \d\Psi(x).
\end{equation}
The exact choice of \( \ff{\:\cdot\:}{s} \) is indicated by the parameter ``\( \:\circ\: \)'' which encompasses many sub-parameters of specific algorithms.
In particular, as we expand on in \cref{sec:quadrature}, both SLQ and KPM can be obtained by particular choices of ``\( \:\circ\: \)''.
Since integrals of polynomials against \( \Psi \) can be evaluated exactly using just the moments of \( \Psi \), such an integral is easily evaluated from the moments of \( \Psi \).

The moments of $\Psi$ through degree $2k$ can be obtained from the Krylov subspace \cite{greenbaum_97,liesen_strakos_13} 
\begin{equation}
    \mathcal{K}_{k}(\vec{A},\vec{v})
    = \operatorname{span}\{\vec{v}, \vec{A}\vec{v}, \ldots, \vec{A}^{k}\vec{v} \}
    = \{ p(\vec{A})\vec{v} : \deg(p) \leq k \}.
\end{equation}
Indeed, for any \( i,j \) with \( 0\leq i,j \leq k \),
\begin{equation}
    (\vec{A}^{i}\vec{v})^\cT (\vec{A}^j\vec{v}) = \vec{v}^\cT \vec{A}^{i+j} \vec{v} = \int_{-\infty}^{\infty} x^{i+j} \d\Psi(x).
\end{equation}
There are many possibilities for how the moment information can be extracted, and different choice of \( \ff{\:\cdot\:}{s} \) (and therefore of \( \qq{\:\cdot\:}{s} \)) often lead to differing approaches. 
However, these approaches are all mathematically equivalent in that  the polynomial moments with respect to one polynomial basis can be converted to moments with respect to a different basis. 
One principled way of doing this is by using using so-called \emph{connection coefficients} \cite{webb_olver_21}.

The resulting approximation to the CSEM $\Phi$ is $\samp{\qq[]{\Psi_\ell}{s}}=\nv^{-1} \sum_{\ell=1}^{\nv} \qq[]{\Psi_\ell}{s}$. 
This naturally yields an approximation $n\int f(x) \d\samp{\qq[]{\Psi_\ell}{s}}(x)$ to the spectral sum $\tr(\fA)$.
The prototypical algorithm considered in this paper can be summarized as follows:
\begin{labelalgorithm}[H]{protoalg}{protoalg}{Prototypical randomized matrix free quadrature}
\begin{algorithmic}[1]
    \Procedure{\thealgorithmname}{$\vec{A}, \nv, k, \circ$}
    \For{ \( \ell = 1, 2, \ldots, \nv \) }
    \State define (implicitly) \( \Psi_\ell \stackrel{\text{iid}}{\sim} \Psi \) by sampling \( \vec{v}_\ell \stackrel{\text{iid}}{\sim} \vec{v} \), \( \EE[\vec{v}\vec{v}^\cT] = n^{-1} \vec{I} \) 
    \State compute the moments of \( \Psi_{\ell} \) through degree \( s \) by constructing \( \mathcal{K}_k(\vec{A},\vec{v}_\ell) \) 
    \State approximate \( \Psi_{\ell} \) by \( \qq{\Psi_{\ell}}{s} \) induced by a polynomial operator \( \ff{\:\cdot\:}{s} \) 
    \EndFor
    \State \Return \( \samp{ \qq{\Psi_{\ell}}{s} } = \frac{1}{\nv} \sum_{\ell=1}^{\nv} \qq{\Psi_\ell}{s} \)

\EndProcedure
\end{algorithmic}
\end{labelalgorithm}

The main purpose of this paper is to study the theoretical and practical behavior difference choices of ``$\:\circ\:$'' for both the tasks of spectral sum and spectrum approximation.
In particular, our theoretical bounds and numerical experiments highlight the qualitative features of KPM and SLQ.

\subsection{Notation}
\label{sec:preliminaries}

Throughout, matrices are denoted by bold uppercase letters and vectors are denoted by bold lowercase letters.
The conjugate transpose of a matrix \( \vec{B} \) is denoted by \( \vec{B}^\cT \).
For \( S\subseteq \R \) and a function \( f: S \to\R \), we define \( \| f \|_{S} = \sup_{x\in S} |f(x)| \).
Integrals are written with the variable of integration suppressed.
Finally, the average of quantities $\{x_\ell\}_{\ell=1}^{\nv}$ is written $\langle x_\ell \rangle = \nv^{-1} \sum_{\ell=1}^{\nv} x_\ell$.

\section{Background}
\label{sec:background}

\subsection{Past theoretical analysis}
\label{sec:related_work}

Using the triangle inequality, we can decompose the error of \cref{alg:protoalg} for spectral sum approximation as
\begin{equation}
    \left|n^{-1} \tr(\fA) - \int f \d \samp{\qq{\Psi_\ell}{s}} \right| 
    \leq 
    \bigg| \int f \d\big(\Phi - \samp{\Psi_\ell}\big) \bigg|
    + \bigg| \int f \d\big(\samp{\Psi_\ell} - \samp{\qq{\Psi_\ell}{s}}\big) \bigg|.
\label{eqn:proto_triangle_f}
\end{equation}
The first term can be bounded by trace estimation bounds and is made small by making the number of samples $m$ of the weighted CESM large. 
The second term can be bounded by analyzing the convergence of quadrature rules, and is made small by making the degree of the approximation $s$ large.
This approach has been applied to various choices of ``$\circ$'' with the aim of balancing $m$ and $s$ so that the error of \cref{eqn:proto_triangle_f} is of size $n^{-1}\varepsilon$ with probability at least $1-\delta$ \cite{han_malioutov_avron_shin_17,ubaru_chen_saad_17,cortinovis_kressner_21}.
Such bounds hold for a \emph{fixed} function $f$.

More recently, spectrum approximation in Wasserstein distance was concurrently analyzed for KPM \cite{braverman_krishnan_musco_22} and SLQ \cite{chen_trogdon_ubaru_21}.
In both cases, it is shown that (up to log factors), to approximate \( \Phi \) in Wasserstein distance to accuracy \( \varepsilon \), one must set $s = O(\varepsilon^{-1})$. 
Amazingly, the number of samples $\nv$ required is constant (e.g. 1) provided (up to log factors) \( \varepsilon \gg n^{-1/2} \) as \( n\to\infty \).
This fact was already known informally to the physics community \cite{girard_87,weisse_wellein_alvermann_fehske_06}.

It is well-known that two distributions are close in Wasserstein distance if and only if the corresponding integrals are close for \emph{all} Lipschitz functions. 
Thus, the bounds for spectrum approximation can be viewed as uniform bounds for Lipschitz functions.
In some applications, one wishes to approximate functions with more regularity than Lipschitz continuity.
For instance, in quantum thermodynamics computing quantities such as $\tr(\exp(-\beta \vec{A}))$ for a range of $\beta$ is of interest; see e.g. \cref{sec:spin}.
In these cases, we might hope for a more favorable dependence of $s$ on $\varepsilon$ than the $s = O(\varepsilon^{-1})$ given by \cite{chen_trogdon_ubaru_21,braverman_krishnan_musco_22}.
To the best of our knowledge, this setting has not been analyzed in the literature.

\subsection{Our contributions}
We provide theoretical bounds in \cref{sec:main_bounds}.
While our bounds for spectrum approximation do improve past bounds slightly, we believe their true value is in the unified analysis which makes the connections between these well-known algorithms clearer.
In addition, we show how to derive bounds that hold simultaneously for all functions in a specified class of functions.
We apply this approach to analytic bounded functions on a Bernstein ellipse and obtain explicit bounds.
To the best of our knowledge, these are the first bounds of this type for trace estimation algorithms and provide a theoretical justification for approximating parameterized functions using matrix-free quadrature algorithms.

In addition to theory, we provide numerical experiments that highlight a number of qualitative trade-offs between the algorithms.
In particular, we compare SLQ and KPM in a range of settings and study the impact of using KPM approximations corresponding to orthogonal polynomial families other than the Chebyshev polynomials. 
We have not seen a thorough study of these topics in the literature.
While it has been long recognized that KPM can be implemented with any orthogonal polynomial family \cite{silver_roder_94,weisse_wellein_alvermann_fehske_06}, we are unaware of any detailed explorations into non-classical families. 
Our numerical experiments demonstrate that moving beyond Chebyshev based methods may result in practical benefits.

\subsection{Other algorithms}
As a consequence of the central limit theorem, the average of iid samples of quadratic trace estimators requires \( O(\varepsilon^{-2}) \) samples to reach accuracy \( \varepsilon \) as $\varepsilon\to 0$. 
In fact, any algorithm which returns a linear combination of estimators depending on vectors drawn independently of \( \vec{A} \) requires \( O(\varepsilon^{-2}) \) samples to obtain an approximation of the trace accurate to within a multiplicative factor \( 1\pm \varepsilon \) \cite{wimmer_wu_zhang_14}.
A number of papers aim to avoid this dependence on the number of samples by computing the trace of a low-rank approximation to \( f(\vec{A}) \) \cite{weisse_wellein_alvermann_fehske_06,lin_16,gambhir_stathopoulos_orginos_17,saibaba_alexanderian_ipsen_17,morita_tohyama_20,meyer_musco_musco_woodruff_21,li_zhu_21,chen_hallman_22}.
Such a technique provides a natural variance reduction if quadratic trace estimators are applied to the remainder. 
This was analyzed in \cite{meyer_musco_musco_woodruff_21} which introduces an algorithm called Hutch++.
Hutch++ provably returns an  estimate of the trace of a positive definite matrix to relative error \( 1\pm \varepsilon \) using just \( O(\varepsilon^{-1}) \) matrix-vector products, and it is shown this \( \varepsilon \) dependence is nearly optimal in certain matrix-vector query models; see also \cite{persson_cortinovis_kressner_22,epperly_tropp_webber_24} for improvements and \cite{chen_hallman_22,persson_kressner_22} for efficient implementations for matrix functions.

\subsection{Quadrature approximations for weighted spectral measures}
\label{sec:quadrature}

Here we define several standard choices for \( \ff{\:\cdot\:}{s} \). In particular, we introduce quadrature by interpolation in \cref{def:iq}, Gaussian quadrature (SLQ) in \cref{def:gq}, and quadrature by approximation in \cref{def:aq}.
Finally, in \cref{def:damped}, we introduce a damped version of quadrature by approximation (KPM) which are used to ensure the positivity of the resulting approximations.
Throughout $\mu$ will be a fixed probability distribution function (e.g. the orthogonality distribution for the Chebyshev polynomials of the first kind).

We do not discuss the implementation of such quadrature rules, as the implementation does not impact the mathematical behavior of the algorithms (and therefore the corresponding bounds).
Details on potential implementations can be found in the literature \cite{skilling_89,bai_fahey_golub_96,weisse_wellein_alvermann_fehske_06,chen_23} or the supplementary materials.

\begin{definition}[Quadrature by interpolation; $\circ = \textrm{i}$]
\label{def:iq}
    The degree \( s \) polynomial interpolating a function \( f \) at the zeros of \( p_{s+1} \), the degree $s+1$ orthogonal polynomial with respect to \( \mu \), is denoted by \( \ff[i]{f}{s} \).
    The corresponding quadrature approximation for $\Psi$ is \( \qq[i]{\Psi}{s} \).
\end{definition}
It's not hard to see that if the zeros of $p_{s+1}$ are  \( \{\theta_{k}\}_{k=1}^{s+1} \), then
\begin{equation}
    \qq[i]{\Psi}{s}(x) = \sum_{k=1}^{s+1} \omega_{k}^{} \bOne(\theta_{k} \leq x),
\end{equation}
where the weights \( \{ \omega_{k}^{} \}_{k=1}^{s+1} \) are chosen such that the moments of \( \qq[i]{\Psi}{s} \) agree with those of \( \Psi \) through degree \( s \).

While interpolation-based quadrature rules supported on \( s \) nodes do not, in general, integrate polynomials of degree higher than \( s-1 \) exactly, if we allow $\mu$ to depend on $\Psi$ we can do better.
\begin{definition}[Gaussian quadrature; $\circ = \textrm{g}$]
\label{def:gq}
The degree \( 2s-1 \) Gaussian quadrature rule \( \qq[g]{\Psi}{2s-1} \) for \( \Psi \) defined as $\qq[i]{\Psi}{s-1}$ where $\mu = \Psi$.
\end{definition}
It is well-known that the $k$-point Gaussian quadrature rule integrate polynomials of degree \( 2s-1 \) exactly \cite[Chapter 6]{golub_meurant_09}.
In the context of matrix-free quadrature, the Gaussian quadrature rule can be obtained by the Lanczos algorithm.

Rather than defining a quadrature approximation using an interpolating polynomial, we might use an approximating polynomial.
\begin{definition}[Quadrature by approximation; $\circ = \textrm{a}$]
\label{def:aq}
The projection of \( f \) onto polynomials through degree \( s \) in the inner product induced by $\mu$ is denoted \( \ff[a]{f}{s} \).
The corresponding quadrature approximation for $\Psi$ is \( \qq[a]{\Psi}{s} \).
\end{definition}

By definition, \(  \ff[a]{f}{s} = \sum_{k=0}^{s} \left(\int f p_k\d\mu \right) \:p_k \).
Then, expanding the integral of $\ff[a]{f}{s}$ against \( \Psi \),
\begin{equation}
    \int \ff[a]{f}{s} \d\Psi
    = \int \sum_{k=0}^{s} \left(\int f p_k \d\mu\right) p_k \d\Psi
    = \int f  \left(\sum_{k=0}^{s} \left(\int p_k \d\Psi\right) p_k \right) \d\mu.
\end{equation}
This implies 
\begin{equation}
    \frac{\d\qq[a]{\Psi}{s}}{\d\mu} 
    = \sum_{k=0}^{s} \left( \int p_k \d\Psi \right) p_k 
    = \sum_{k=0}^{s} m_k p_k,
\end{equation}
where \( \!\d\qq[a]{\Psi}{s}/\!\d\mu \) is the Radon--Nikodym derivative of \( \qq[a]{\Psi}{s} \) with respect to \( \mu \).
Alternately, supposing%
\footnote{If \( \Psi = \Psi(\vec{A},\vec{v}) \) then \( \Psi \) is not absolutely continuous with respect to the Lebesgue measure (or any equivalent measure) so the Radon--Nikodym derivative does not exist. 
However, there are absolutely continuous distributions with the same modified moments as \( \Psi \) up to arbitrary degree, so conceptually one can use such a distribution instead.}
that the Radon--Nikodym derivative \( \!\d\Psi/\!\d\mu \) exists, we observe 
\begin{equation}
    \frac{\d\Psi}{\d\mu}
    = \sum_{k=0}^{\infty} \left( \int p_k \frac{\d\Psi}{\d\mu}\d\mu \right) \: p_k
    =
    \sum_{k=0}^{\infty} \left( \int p_k \d\Psi \right) \: p_k.
\end{equation}
In other words, at least formally, \( \!\d\qq[a]{\Psi}{s}/\!\d{x} \) is the polynomial approximation to \( \!\d\Psi/\!\d\mu \); i.e. \( \!\d\qq[a]{\Psi}{s}/\!\d{\mu} = \ff[a]{\!\d\Psi/\!\d\mu}{s} \).
Thus, the KPM can be viewed as arising from a polynomial approximation to the Radon--Nikodym derivative $\d\mu/\d x$.

It is common to use a damped version of \( \qq[a]{\Psi}{s} \) to obtain weakly increasing quadrature rules.
Damping involves replacing $p_k$ with $\rho_k p_k$ for $k=0,1,\ldots, s$ in \( \ff[]{\,\cdot\,}{s} \), where $\{\rho_k\}_{k=0}^{s}$ are the so-called damping coefficients.
\begin{definition}[damped quadrature by approximation; $\circ = \textrm{d}$]
\label{def:damped}
    Given damping coefficients \( \{ \rho_k \}_{k=0}^{s} \), the damped approximant \( \ff[d]{f}{s} \) is defined by
\begin{equation*}
    \ff[d]{f}{s}(x) = \int P_x \ff[a]{f}{s} \d\mu,
    \qquad P_x(y) 
    = \sum_{k=0}^{s} \rho_{k} p_k(x) p_k(y).
\end{equation*}
The corresponding quadrature approximation is \( \qq[d]{\Psi}{s} \).
\end{definition}

Using that $p_k(x)$ is orthogonal to $1$ with respect to $\mu$, $\int P_x \d\mu = \rho_{0}$.
Thus, \( \qq[d-i]{\Psi}{s} \) and \( \qq[d]{\Psi}{s} \) will have unit mass provided \( \rho_0 = 1 \).
Next, suppose \( P_x(y) \geq 0 \) for all \( x,y  \) and that \( f \geq 0 \).
Then using that $\int P_x \ff[a]{f}{s} \d\mu = \int P_x f \d\mu \geq 0$, \( \int f \d\qq[d]{\Psi}{s} \geq 0 \) so the approximation is also weakly increasing. 

One particularly important damping kernel for the case that $\mu$ is the orthogonality distribution for the Chebyshev-$T$ polynomials (see \cref{def:cheb}) is given by the Jackson coefficients
\begin{equation}
\label{eqn:jackson_coeffs}
    \rho_{k}^{J}
    = ({s+2})^{-1} \left( (s-k+2)\cos \left( \frac{k \pi}{s+2} \right) + \sin \left( \frac{k \pi}{s+2} \right)\cot \left( \frac{\pi}{s+2} \right) \right)
    ,\qquad
    k=0,1,\ldots, s.
\end{equation}
The associated damping kernel $P_x$ is non-negative and was used in the original proof of Jackson's theorem \cite{jackson_12}. 
It leads to the Jackson damped KPM approximation, which is the most popular KPM variant \cite{silver_roeder_voter_kress_96,weisse_wellein_alvermann_fehske_06,braverman_krishnan_musco_22}.
The rate of convergence of polynomial approximations using these damping coefficients is estimated in \cref{thm:damped_cheb} below.
For a discussion on other damping schemes, we refer readers to \cite{weisse_wellein_alvermann_fehske_06}.

\section{Main bounds}
\label{sec:main_bounds}

In this section, we provide theoretical bounds for several instances of the protoalgorithm, \cref{alg:protoalg}.
We prioritize simplicity of our result statements. 
In particular, in several cases it is clear from the proof that constants can be improved marginally over what is stated.
We do not analyze the case $\circ = \mathrm{i}$, although it is clear from the analysis how a similar approach could be applied to obtain bounds. 

We will require the definition of Lipshitz functions:
\begin{definition}
    We say that \( f:\R^m\to\R^n \) is $L$-Lipschitz on $S\subset\mathbb{R}^m$ if \( \|f(\vec{x}) - f(\vec{y})\| \leq L\|\vec{x}-\vec{y}\| \) for all \( \vec{x},\vec{y} \in S \).
\end{definition}

For (damped) quadrature by approximation, our theoretical results focus on the case that $\mu$ is a shifted and scaled version of the orthogonality measure for the Chebyshev polynomials of the first kind:
\begin{definition}\label{def:cheb}
    The Chebyshev distribution function of the first kind, \( \mu_{a,b}^T : [a,b] \to [0,1] \), is
    \begin{equation*}
        \mu_{a,b}^T(x) = \int_{a}^{x}\frac{2}{\pi(b-a)} \frac{1}{\sqrt{1-\big(\frac{2}{b-a}t - \frac{b+a}{b-a}\big)^2}} \d{t}
        = \frac{1}{2} + \frac{1}{\pi}\arcsin\left(\frac{2}{b-a}x - \frac{b+a}{b-a}\right).
    \end{equation*}
\end{definition}
This is by far the most common choice of $\mu$ in the literature \cite{skilling_89,silver_roeder_voter_kress_96,weisse_wellein_alvermann_fehske_06,lin_saad_yang_16,braverman_krishnan_musco_22}, although, as we demonstrate with numerical experiments in \cref{sec:numerical_experiments}, other choices of $\mu$ can have good qualitative properties.

We also define the Chebyshev polynomials of the first kind:
\begin{definition}
    The Chebyshev polynomials of the first kind, denoted \( \{ T_i \}_{i=0}^{\infty} \), are defined by the recurrence \( T_0(x) = 1 \), \( T_1(x) = x \), and
    \begin{equation}
        T_{i+1}(x) := 2 x T_i(x) - T_{i-1}(x)
        ,\qquad i=1,2,\ldots.
    \end{equation}
\end{definition}

For simplicity, we will assume that the random vectors are drawn from the uniform distribution on the unit hypersphere. \begin{definition}
Let $\operatorname{Unif}(\mathbb{S}^{n-1})$ be the uniform distribution on unit-length vectors in $\mathbb{R}^n$.
\end{definition}

In this case, the trace estimator exhibits sub-Gaussian concentration. 
We follow the general approach of \cite{popescu_short_winter_06}.
\begin{lemma}\label{thm:subgaussian}
There exists an absolute constant $c>0$ such that for any $n\times n$ matrix $\vec{B}$ and $\varepsilon>0$, if $\{\vec{v}_\ell\}_{\ell=1}^{\nv}$ are iid samples from $\operatorname{Unif}(\mathbb{S}^{n-1})$, then
\begin{equation*}
        \PP \Big[
        \big| n^{-1} \tr(\vec{B}) - \samp{\vec{v}_\ell^\cT \vec{B} \vec{v}_\ell} \big|
        > \varepsilon
        \Big]
        \leq 2\exp \left( - \frac{c n \nv \varepsilon^2 }{\|\vec{B}\|_2^2}\right).
    \end{equation*}
\end{lemma}

\begin{proof}
    Define $F(\vec{x}) = \vec{x}^\cT\vec{B}\vec{x}$ and observe that 
    \begin{equation}
        \big| F(\vec{x}) - F(\vec{y}) \big|
        = \big|\vec{x}^\cT\vec{B}\vec{x} - \vec{y}^\cT \vec{B}\vec{y}\big|
        = \frac{1}{2}\big| (\vec{x}+\vec{y})^\cT \vec{B} (\vec{x}-\vec{y}) + (\vec{x}-\vec{y})^\cT \vec{B} (\vec{x}+\vec{y}) \big|.
    \end{equation}
    We also have that 
    \begin{equation}
        \big|(\vec{x}-\vec{y})^\cT \vec{B} (\vec{x}+\vec{y}) \big|
        \leq \| \vec{x} + \vec{y} \| \|\vec{B} \| \| \vec{x}-\vec{y}\|
        \leq (\| \vec{x} \| + \| \vec{y} \|) \|\vec{B} \|_2 \| \vec{x}-\vec{y}\|
        \leq 2 \|\vec{B}\|_2 \|\vec{x}-\vec{y}\|.
    \end{equation}
    Therefore $F(\vec{x})$ is $2 \|\vec{B}\|_2$-Lipshitz on $\mathbb{S}^{d-1}$.
    By the concentration of Lipshitz functions on the sphere (see for instance \cite[Theorem 5.1.4]{vershynin_18}), there is an absolute constant $c$ such that, if $\vec{v}\sim \operatorname{Unif}(\mathbb{S}^{n-1})$,
    \begin{equation*}
        \PP\bigl[ | F(\vec{v}) - \EE[F(\vec{v})] | > \varepsilon \bigr]
        \leq 2 \exp \left( \frac{-c n\varepsilon^2}{2\|\vec{B}\|_2^2} \right).
    \end{equation*}
    The final result holds by a standard bound for the average of iid copies of a sub-Gaussian random variable (see for instance \cite[\S2.5 and Theorem 2.6.2]{vershynin_18}), observing that $\EE[F(\vec{v})] = n^{-1}\tr(\vec{B})$, and relabeling $c$.    
\end{proof}

For Rademacher (uniform $\pm 1$) and Gaussian vectors, sharper bounds for small $\varepsilon$ based on Hanson--Wright type  inequalities are also known \cite{meyer_musco_musco_woodruff_21,cortinovis_kressner_21}.
These bounds typically improve the dependence on $n\|\vec{B}\|_2^2$ to $\|\vec{B}\|_\F^2$.
The precise bounds/distribution used are not that relevant to the big picture of this paper, and it is clear from our analysis how other bounds can be substituted.

As an immediate corollary of \cref{thm:subgaussian} we have the following (trivial) result on the Chebyshev moments of $\Phi$ and $\samp{\Psi_{\ell}}$.
\begin{lemma}
\label{thm:cheb_union}
Fix $\varepsilon>0$ and suppose $\{\vec{v}_\ell\}_{\ell=1}^{\nv}$ are iid samples from $\operatorname{Unif}(\mathbb{S}^{n-1})$. Let $c$ as in \cref{thm:subgaussian}. 
Then, for any $s>0$,
\begin{equation*}
    \PP\left[\exists k\in\{0,1,\ldots, s\}:  \left| \int T_k \d\left(\Phi - \samp{\Psi_\ell} \right) \right| > \varepsilon \right]
    \leq 2 s \exp \left( - c \nv n\varepsilon^2 \right).
\end{equation*}
\end{lemma}

\begin{proof}
Note that $\int T_0 \d(\Psi-\samp{\Psi_\ell}) = 0$, since $\Psi$ and $\samp{\Psi_\ell}$ are both probability distributions.
Fix $k\in\{1,2,\ldots, s\}$.
Using that $\tr(\fA) = n \int f\d\Phi$ and that $\| T_k(\vec{A}) \|_2\leq \| T_k \|_{[-1,1]} = 1$, \cref{thm:subgaussian} implies that 
\begin{equation}
    \PP\left[ \left| \int T_k \d\left(\Phi - \samp{\Psi_\ell} \right) \right| > \varepsilon \right]
    \leq 2 \exp \left( - cn\nv\varepsilon^2 \right).
\end{equation}
Using a union bound over the events that the $k$-th moment is large for each $k=1,2,\ldots, s$ and absorbing the $n$ into $\varepsilon$ gives the result.
\end{proof}

Throughout we make use of the following key lemma which is inspired by the approach taken in \cite{braverman_krishnan_musco_22}:
\begin{lemma}\label{thm:general_unif}
Suppose $\Upsilon_1$ and $\Upsilon_2$ are two probability distribution functions supported on $[-1,1]$.
Then, for any $s>0$, and for all bounded $f$,
\begin{equation*}
\left| \int f \d \left( \Upsilon_1 - \Upsilon_2 \right) \right|
\leq 
2 \| f - \ff[]{f}{s} \|_{[-1,1]}  + 2\sum_{k=1}^{s} \left| \int \ff{f}{s} T_k\d\mu_{-1,1}^T \right| \left| \int T_k \d \left( \Upsilon_1 - \Upsilon_2 \right)  \right|.
\end{equation*}
\end{lemma}

\begin{proof}
Fix a bounded function $f$. 
We begin with the main case.
By the triangle inequality,
\begin{equation}
    \left| \int f \d \left( \Upsilon_1 - \Upsilon_2 \right) \right|
    \leq \left| \int (f- \ff[]{f}{s}) \d \left( \Upsilon_1 - \Upsilon_2 \right)  \right|
    +\left| \int \ff[]{f}{s} \d \left( \Upsilon_1 - \Upsilon_2 \right)  \right|.
\end{equation}
Since $\Upsilon_1$ and $\Upsilon_2$ are probability distribution functions,
\begin{equation}
\label{eqn:fdiffbd}
    \left| \int (f- \ff[]{f}{s}) \d \left( \Upsilon_1 - \Upsilon_2 \right)  \right|
    \leq 
    \| f- \ff[]{f}{s} \|_{[-1,1]}  \int  |\!\d \left( \Upsilon_1 - \Upsilon_2 \right) |
    \leq 2\| f- \ff[]{f}{s} \|_{[-1,1]}. 
\end{equation}
Now, note that 
\begin{equation}
\ff[]{f}{s}
= \mathop{\smash{\sideset{}{'}\sum}\vphantom{\sum}}\limits_{k=0}^{s} \left(2 \int \ff[]{f}{s} T_k \d\mu_{-1,1}^T\right)  T_k
,
\end{equation}
where the prime indicates that the $k=0$ term in the sum is multiplied by $1/2$.
Using that $\Upsilon_1$ and $\Upsilon_2$ are both probability distributions and then applying the triangle inequality we find
\begin{equation}
    \left| \int \ff[]{f}{s} \d \left( \Upsilon_1 - \Upsilon_2 \right)  \right|
    \leq \sum_{k=1}^{s} \left| 2 \int \ff[]{f}{s} T_k \d\mu_{-1,1}^T \right|  \left| \int T_k \d \left( \Upsilon_1 - \Upsilon_2 \right)  \right|.
\end{equation}
\end{proof}

If uniform bounds for $\|f - \ff{f}{s}\|_{[-1,1]}$ and the Chebyshev coefficients of $\ff{f}{s}$ are known, then \cref{thm:general_unif} allows us to bound the closeness of two quadrature approximations in terms of the difference of their Chebyshev moments. 
This is the approach we will use in the next sections.

\subsection{Spectrum approximation}
\label{sec:spectral_density}

We begin with the task of spectrum approximation. 
Our main result is stated in \cref{thm:spec_approx}, which we prove following closely the approach of \cite{braverman_krishnan_musco_22}.
Compared to the analysis of KPM in \cite{braverman_krishnan_musco_22}, we obtain slightly better constants.
Compared to the analysis of SLQ in \cite{chen_trogdon_ubaru_21}, we replace a dependence of the failure probability on $n$ with a dependence on $s$, but obtain slightly worse constants for $s$ in terms of the accuracy $\varepsilon$.
The main advantage over each is a unified analysis.

Our analysis of spectrum approximation targets Wasserstein (earthmover) distance:
\begin{definition}
\label[definition]{def:wasserstein}
Let \( \Upsilon_1 \) and \( \Upsilon_2 \) be two probability distribution functions.
The Wasserstein distance between \( \Upsilon_1 \) and \( \Upsilon_2 \), denoted \( \W(\Upsilon_1,\Upsilon_2) \), is defined by
\begin{equation*}
    \W(\Upsilon_1,\Upsilon_2) = \int | \Upsilon_1(x) - \Upsilon_2(x) | \d{x}.
\end{equation*}
\end{definition}

The Wasserstein distance allows comparison of discrete and continuous probability distributions and has been studied in a number of past works on spectrum approximation
\cite{kong_valiant_17,cohensteiner_kong_sohler_valiant_18,chen_trogdon_ubaru_21,braverman_krishnan_musco_22,bhattacharjee_jayaram_musco_musco_ray_24}, 
However, in some situations, other metrics may be more meaningful. For instance, if it is important for two distribution functions to agree to very high precision in a certain region, but only to moderate accuracy in others, then the Wasserstein distance may be unsuitable.

It is well-known that the Wasserstein distance between two distribution functions is closely related to the difference of their integrals on Lipschitz functions; see for instance \cite[Remark 6.5]{villani_09}.
\begin{lemma}
\label[lemma]{thm:wasserstein_1lip}
    Suppose \( \Upsilon_1 \) and \( \Upsilon_2 \) are two probability distribution functions each constant on \( (-\infty,a) \) and \( (b,\infty) \).
Then,
\begin{equation*}
    \W(\Upsilon_1,\Upsilon_2) = \sup \left\{ \int f \d \big( \Upsilon_1 - \Upsilon_2\big) : \text{$f$ is 1-Lipschitz on $[a,b]$} \right\}.
\end{equation*}
\end{lemma}

The Jackson damped Chebyshev approximant satisfies a bound similar to Jackson's theorem for best polynomial approximation, and has decaying Chebyshev coefficients:
\begin{theorem}\label{thm:damped_cheb}
    Suppose \( f \) is 1-Lipschitz on \( [-1,1] \), \( \mu  = \mu_{-1,1}^T \), and we use Jackson's damping coefficients as in \cref{eqn:jackson_coeffs}.
    Then, for all $s>0$ and $0< k \leq s$,
    \begin{equation*}
    \| f - \ff[d]{f}{s} \|_{[-1,1]} \leq \frac{6}{s}
    ,\qquad
        \left|2 \int \ff[d]{f}{s} T_k \d\mu_{-1,1}^T \right|  \leq \frac{4}{\pi k}.
    \end{equation*}
\end{theorem}

\begin{proof}
The first result is due to a common proof of Jackson's theorem \cite[Theorem 1.4]{rivlin_81}.
By definition, we have $\int \ff[d]{f}{s} T_k \d\mu_{-1,1}^T = \rho_k^J\int f T_k \d\mu_{-1,1}^T$, and $|\rho_k^J|\leq 1$. 
Since $f$ is of bounded variation $2$ \cite[Theorem 7.1]{trefethen_19} asserts $|2\int f T_k \d\mu_{-1,1}^T|\leq 4(\pi k)^{-1}$ giving the second result. 
\end{proof}

The best polynomial approximation satisfies a similar bound $\min_{\deg(p) \leq s} \| f - p \|_{[-1,1]} \leq (\pi/2) (s+1)^{-1}$, and the constant $\pi/2$ is the best possible; see \cite[Section 87]{achieser_92} and \cite[Chapter 4]{cheney_00} for details.
Thus, \cref{thm:damped_cheb} gives a nearly optimal approximation.

We are now prepared to state our main theorem for spectrum approximation. 
\begin{theorem}\label{thm:spec_approx}
Let $c$ be as in \cref{thm:subgaussian}.
Suppose $\circ =  \textup{g}$ or $\circ = \textup{d}$ where $\mu = \mu_{\lmin,\lmax}^T$ and Jackson's damping is used.
Fix $\varepsilon>0$ and $\delta\in(0,1)$. 
Suppose
\begin{equation*}
s \geq \frac{36}{\varepsilon}
,\qquad
m \geq \frac{7 (\ln(s)+1)^2}{c n  \varepsilon^2}\ln\left(\frac{2s}{\delta}\right),
\end{equation*}
and $\{\vec{v}_\ell\}_{\ell=1}^{\nv}$ are iid samples from $\operatorname{Unif}(\mathbb{S}^{n-1})$.
Then
the output $\samp{\qq{\Psi_\ell}{s}}$ of \cref{alg:protoalg} satisfies, 
\begin{align*}
    \PP\big[\W(\Phi,\samp{\qq[]{\Psi_\ell}{s}})
    >  (\lmax-\lmin) \varepsilon \big]
    \leq 
    \delta.
\end{align*}
\end{theorem}

\begin{proof}
Without loss of generality, assume $\lmin=-1$ and $\lmax=1$.
Since $\samp{\Psi_{\ell}}$ is a probability distribution\footnote{While $\samp{\Psi_{\ell}}$ is random, since we assume $\|\vec{v}_\ell\|=1$, it is (almost) surely a probability distribution.}, inserting \cref{thm:damped_cheb} into \cref{thm:general_unif} we obtain a bound for all 1-Lipschitz $f$:
\begin{equation}\label{eqn:cheb_wass_near}
    \left| \int f \d \left( \Phi - \samp{\Psi_{\ell}} \right) \right|
    \leq 
    2 \left(\frac{6}{s}\right)  + \sum_{k=1}^{s} \frac{4}{\pi k} \left| \int T_k \d \left( \Phi - \samp{\Psi_{\ell}} \right)  \right|.
\end{equation}
Provided $\nv \geq (cn\eta^2\varepsilon^2)^{-1}\ln(2s/\delta)$ for some $\eta>0$ to be determined, \cref{thm:cheb_union} gives the bound 
\begin{equation}
\label{eqn:cheb_union_gq}
    \PP\left[ \exists k\in\{0,1,\ldots, s\}: \left| \int T_k \d\left(\Phi - \samp{\Psi_\ell} \right) \right| > \eta \varepsilon \right]
    \leq \delta.
\end{equation}
Now, note that $\int T_k \d\qq[g]{\Psi_\ell}{s} = \int T_k\d\Psi_\ell$ for $k\leq s$, so we can replace $\samp{\Psi_\ell}$ with $\samp{\qq[g]{\Psi_\ell}{s}}$ in \cref{eqn:cheb_wass_near,eqn:cheb_union_gq}.
Thus \cref{thm:wasserstein_1lip}, and the fact that $(1+1/2+\cdots + 1/s) \leq \ln(s)+1$ imply
\begin{equation}\label{eqn:Wphigq}
\PP\Big[\W(\Phi,\samp{\qq[g]{\Psi_\ell}{s}})
    > 12 s^{-1} + 4\pi^{-1} (\ln(s) + 1) \eta \varepsilon  \Big]
    \leq 
    \delta.
\end{equation}
We set $12 s^{-1} \leq \varepsilon/2$ and $4\pi^{-1}(\ln(s)+1)\eta\varepsilon = \varepsilon/2$.
This gives sufficient conditions $s \geq 24\pi^2\varepsilon^{-1}$ and $\eta = (\pi/8)( \ln(s)+1)^{-1} $.
Plugging in this value of $\eta$ into the bound for $m$ and noting that $64/\pi^2<7$ gives the stated bound for $\nv$.

When $\circ = \textup{d}$, observe that if $f$ is 1-Lipschitz, \cref{thm:damped_cheb} implies
\begin{equation}
\left|\int f \d\left( \samp{\Psi_\ell} - \samp{\qq[d]{\Psi_\ell}{s}} \right)\right|
= \left|\int (f - \ff[d]{f}{s}) \d \samp{\Psi_\ell} \right|
\leq \int \| f - \ff[d]{f}{s} \|_{[-1,1]} \d \samp{\Psi_\ell}
\leq \frac{6}{s}.
\end{equation}
Using \cref{thm:wasserstein_1lip}, this implies \( \W(\samp{\Psi_\ell},\samp{\qq[d]{\Psi_\ell}{s}}) \leq 12 s^{-1} \).
Then, we again use \cref{eqn:cheb_wass_near} to obtain a tail bound for $\W(\Phi,\samp{\Psi_\ell})$ analogously to \cref{eqn:Wphigq}.
Applying the triangle inequality gives
\begin{equation}\label{eqn:phipsil_close}
\PP\left[\W(\Phi,\samp{\qq[d]{\Psi_\ell}{s}})
    > 6s^{-1} + 12s^{-1} + 4\pi^{-1} (\ln(s) + 1) \eta\varepsilon  \right]
    \leq 
    \delta.
\end{equation}
We set $(6+12)s^{-1} \leq \varepsilon/2$ and $4\pi^{-1}(\ln(s)+1)\eta \varepsilon = \varepsilon/2$.
This gives sufficient conditions $s \geq 36\varepsilon^{-1}$ and $\eta = (\pi/8)( \ln(s)+1)^{-1} $.

Finally, to remove our assumption $\lmin=-1$ and $\lmax=1$, note that linearly scaling two distribution functions by a factor of $t$ scales the Wasserstein distance by a factor of $t$.
\end{proof}

Note that\cref{thm:spec_approx} implies that for \( \varepsilon \gg n^{-1/2} \) as \( n\to\infty \), a \emph{single} test vector is sufficient approximate the CESM \( \Phi \) in Wasserstein distance to accuracy \( \varepsilon \).

\subsection{Uniform spectral sum approximation}

We now study the task of spectral sum approximation by using \cref{thm:general_unif} to obtain a bound for bounded analytic functions.
It is apparent that the approach we take can be extended to other classes of functions, for instance differentiable functions.

\begin{definition}
    For \( \rho > 1 \) denote by \( E_\rho(a,b) \) the region bounded by the ellipse centered at \( \frac{a+b}{2} \) with semi-axis lengths \( \frac{b-a}{2} \frac{1}{2} (\rho + \rho^{-1}) \) and \( \frac{b-a}{2}\frac{1}{2} (\rho-\rho^{-1}) \) along the real and imaginary directions; i.e
    \begin{equation*}
        E_\rho(a,b) = \left\{ z \in\mathbb{C} : z = \frac{b-a}{2} \frac{1}{2} (u + u^{-1}) + \frac{a+b}{2}, \: u = r \exp(i\theta), \: r\in[0,\rho), \: \theta\in[0,2\pi) \right\}. 
    \end{equation*}
\end{definition}

\begin{definition}\label{def:analytic}
    We denote by $\mathsf{Anl}(M,\rho,a,b)$ the set of analytic functions bounded in modulus by $M$ on $E_\rho(a,b)$.
\end{definition}

Standard results in approximation theory bound the convergence of the Chebyshev approximant and its coefficients.
\begin{theorem}\label{thm:analytic}
Suppose $f\in\mathsf{Anl}(M,\rho,-1,1)$ and $\mu = \mu_{-1,1}^T$. 
Then, for all $s>0$
\[
\|f - \ff[a]{f}{s} \|_{[-1,1]}
\leq \frac{2M\rho^{-s}}{\rho-1}
,\qquad
\left|2 \int f T_s \d\mu_{-1,1}^T \right|
\leq 2M\rho^{-s}.
\]
\end{theorem}

\begin{proof}
The results are taken from \cite[Theorem 8.2]{trefethen_19} and \cite[Theorem 8.1]{trefethen_19} respectively.
\end{proof}

\begin{theorem}
\label{thm:unif_anl}
Let $c$ as in \cref{thm:subgaussian}.
Fix constants $M>0$ and $\rho > 1$. 
Suppose
\begin{equation*}
s\geq \ln \left( \frac{8M}{(\rho-1)\varepsilon} \right) \frac{1}{\ln(\rho)},
\qquad
m \geq \frac{16M^2}{c (\rho-1)^2 n\varepsilon^2}\ln\left(\frac{2s}{\delta}\right),
\end{equation*}
and $\{\vec{v}_\ell\}_{\ell=1}^{\nv}$ are iid samples from $\operatorname{Unif}(\mathbb{S}^{n-1})$.
Then
the output $\samp{\qq[g]{\Psi_\ell}{s}}$ of \cref{alg:protoalg} satisfies, 
\begin{equation*}
    \PP\left[ 
    \exists f\in\mathsf{Anl}(M,\rho,\lmin,\lmax) :
    \left| n^{-1}\tr(\fA) - \int f \d\samp{\qq[g]{\Psi_\ell}{s}} \right|> \varepsilon
    \right]
    \leq \delta.
\end{equation*}
\end{theorem}

\begin{proof}
Note that $f\in \mathsf{Anl}(M,\rho,a,b)$, then $\tilde{f}\in \mathsf{Anl}(M,\rho,-1,1)$, where $\tilde{f}(x) = f(\frac{b-a}{2}x + \frac{a+b}{2})$.
Thus, we can take a general symmetric matrix $\vec{A}$ and map it to $\tilde{\vec{A}} = \frac{2}{b-a}\vec{A} - \frac{a+b}{b-a}\vec{I}$ whose spectrum is contained in $[-1,1]$.
Then $f(\vec{A}) = \tilde{f}(\tilde{\vec{A}})$ and $\int f \d\samp{\qq[g]{\Psi_\ell}{s}} = \int \tilde{f} \d\samp{\qq[g]{\tilde{\Psi}_\ell}{s}}$.
Thus, it suffices to assume $\lmin=-1$ and $\lmax=1$.

Inserting \cref{thm:analytic} 
into \cref{thm:general_unif} (with $\Upsilon_1 = \Phi$, $\Upsilon_2 = \samp{\qq[g]{\Psi_\ell}{s}}$, and $\ff{f}{s}$ the Chebyshev approximant bounded in \cref{thm:analytic}) and applying the triangle inequality over the average, we obtain a bound for all $f\in\mathsf{Anl}(M,\rho,-1,1)$:
\begin{align}
    \left| \int f \d \left( \Phi - \samp{\qq[g]{\Psi_\ell}{s}} \right) \right|
    &\leq 
    2 \| f - \ff[a]{f}{s} \|_{[-1,1]}  + 2\sum_{k=1}^{s} \left| \int \ff[a]{f}{s} T_k\d\mu_{-1,1}^T \right| \left| \int T_k \d \left( \Phi - \samp{\qq[g]{\Psi_\ell}{s}} \right)  \right|
    \nonumber\\&\leq 2\left(\frac{2M\rho^{-s}}{\rho-1}\right) + \sum_{k=1}^{s} 2M\rho^{-k} \left| \int T_k \d \left( \Phi - \samp{\qq[g]{\Psi_\ell}{s}} \right)  \right|.
    \label{eqn:analytic_moments}
\end{align}

Provided $\nv \geq (cn\eta^2\varepsilon^2)^{-1}\ln(2s/\delta)$, \cref{thm:cheb_union} and the fact that the moments of $\samp{\Psi_\ell}$ and $\samp{\qq[g]{\Psi_\ell}{s}}$ agree through degree $s$ gives the bound
\begin{equation}
\label{eqn:cheb_unif_analytic}
    \PP\left[ \exists k\in\{0,1,\ldots, s\}: \left| \int T_k \d\left(\Phi - \samp{\qq[]{\Psi_\ell}{s}} \right) \right| > \eta\varepsilon \right]
    \leq \delta.
\end{equation}
Combining \cref{eqn:analytic_moments,eqn:cheb_unif_analytic} and using that $\sum_{k=1}^{s} \rho^{-k} \leq \sum_{k=1}^{\infty} \rho^{-k} = 1/(\rho-1)$, we obtain a bound
\begin{equation}
\PP\left[ \exists f\in\mathsf{Anl}(M,\rho,-1,1) :
 \left| \int f \d \left( \Phi - \samp{\qq[]{\Psi_\ell}{s}} \right) \right| 
> 
\frac{4M\rho^{-s}}{\rho-1} + \frac{2M \eta\varepsilon}{\rho-1}
\right]
\leq \delta.
\end{equation}
Setting ${4M \rho^{-s}}/({\rho-1})\leq \varepsilon/2$ and $2M\eta\varepsilon/(\rho-1) = \varepsilon/2 $ gives the conditions $s \geq \ln(8M / ((\rho-1)\varepsilon ) ) / \ln(\rho)$ and $\eta = (\rho-1) / (4M)$.
Plugging in this value of $\eta$ into the bound for $m$ gives the stated bound for $\nv$.
Finally, we note that $\tr(\fA) = n\int f\d\Phi$ by definition.
\end{proof}

Note that the number of samples $m$ required by \cref{thm:unif_anl} depends on $s$.
This is in contrast to bounds such as those in \cite{ubaru_chen_saad_17} based on the triangle inequality \cref{eqn:proto_triangle_f} which do not depend on $s$.
However, the benefit of \cref{thm:unif_anl} is that it holds simultaneously for all functions in $\mathsf{Anl}(M,\rho,\lmin,\lmax)$ rather than just a single function. 

One may hope to derive similar bounds for $\circ = \mathrm{a}$ or $\circ = \mathrm{i}$ for which the low-degree moments of $\samp{\Psi_\ell}$ and $\samp{\qq[]{\Psi_\ell}{s}}$ match.
However, the quadrature rules $\samp{\qq{\Psi_\ell}{s}}$ may not be probability distributions, so \cref{thm:general_unif} cannot be applied.
If the total variation of $\samp{\qq{\Psi_\ell}{s}}$ could be bounded, then a bound similar to \cref{thm:general_unif} could be used. 
However, we are unaware of how obtain a priori bounds.
On the other hand, the same technique could not be used to derive similar bounds for Jackson's damped KPM since the moments do not match. 
In fact, for Jackson's damped KPM, $s$ cannot depend logarithmically on $\varepsilon$ since the damped polynomial
approximations converge only algebraically; see e.g. \cref{fig:GQ_CC_tre08}.

In many situations, we may have some set of functions of interest for which we would like to apply \cref{thm:unif_anl}. 
If these functions are analytic on some Bernstin ellipse $E_\rho(a,b)$, then we can obtain bounds for the family by scaling the functions such that they all live in $\mathsf{Anl}(1,\rho,a,b)$.
We now illustrate this for the family $\{ \exp(-\beta x) : \beta \geq 0\}$.

\begin{corollary}\label{thm:exp_unif}
Let $c$ as in \cref{thm:subgaussian} and suppose $\lmax>\lmin>0$.
Let $\kappa:=\lmax/\lmin$ and suppose
\begin{equation*}
s\geq \sqrt{\kappa} \ln \left( \frac{4\sqrt{\kappa}}{\varepsilon} \right)  ,
\qquad
m \geq \frac{4\kappa}{c  n\varepsilon^2}\ln\left(\frac{2s}{\delta}\right),
\end{equation*}
and $\{\vec{v}_\ell\}_{\ell=1}^{\nv}$ are iid samples from $\operatorname{Unif}(\mathbb{S}^{n-1})$.
Then
the output $\samp{\qq[g]{\Psi_\ell}{s}}$ of \cref{alg:protoalg} satisfies, 
\begin{equation*}
    \PP\left[ 
    \exists \beta \geq 0 :
    \left| n^{-1}\tr(\exp(-\beta\vec{A}) - \int \exp(-\beta x) \d\samp{\qq[g]{\Psi_\ell}{s}} \right|> \varepsilon
    \right]
    \leq \delta.
\end{equation*}
\end{corollary}

\begin{proof}
    Define $\rho := (\sqrt{\kappa}+1)/(\sqrt{\kappa}-1)$.
    Note that $\exp(-\beta z)$ is entire and $|\exp(-\beta z)| = \exp(-\beta\operatorname{Re}(z))$.
    The leftmost point of the closure of $E_\rho(\lmin,\lmax)$ is 
    \[
    z^\star := \frac{\lmax-\lmin}{2} \frac{(-\rho)+(-\rho)^{-1}}{2} + \frac{\lmin+\lmax}{2}
    = 0.
    \]
    Hence, the supremum of $|\exp(-\beta z)|$ over $E_\rho(\lmin,\lmax)$ is $|\exp(-\beta z^\star)| = 1$. 
    Therefore, 
    \[
    \exp(-\beta x) \in \mathsf{Anl}(1,\rho,\lmin,\lmax).
    \]
    Now, note that 
    \[
    \frac{1}{\ln(\rho)} = \ln\left(\frac{\sqrt{\kappa}+1}{\sqrt{\kappa}-1}\right)^{-1}\leq \frac{\sqrt{\kappa}}{2}
    ,\qquad
    \frac{1}{\rho-1} = \frac{\sqrt{\kappa}-1}{2} 
    \leq \frac{\sqrt{\kappa}}{2}.
    \]
    Thus, applying \cref{thm:unif_anl} and inserting these bounds gives the result.
\end{proof}

We note that sharper bounds (e.g. for which the accuracy depends on $\beta$) for this family of functions are likely possible. 
We have focused on optimizing the simplicty of the statement of the result rather than its sharpness.

\section{Numerical experiments}
\label{sec:numerical_experiments}

We now provide a number of numerical experiments. 
While several of the examples (e.g. \cref{fig:GQ_CC_tre08,fig:Kneser_convergence,fig:spin_partition}) relate specifically to the theory developed in \cref{sec:main_bounds}, many are chosen to highlight subtle features of variants of \cref{alg:protoalg} which are beyond the current existing theory.
In particular, we aim to illustrate 
(i) Gaussian quadrature often does better than bounds based on best approximation over a single interval $[\lmin,\lmax]$, and
(ii) it can be useful to use $\mu$ other than the orthogonality distribution for the Chebyshev polynomials of the first kind.

In the numerical experiments, we will often compare algorithms based on the number of matrix-vector products used. 
Given $k$ matrix-vector products, one can construct $\mathcal{K}_{k}(\vec{A},\vec{v})$ and obtain the moments of $\Psi$ through degree $2k$.
This can be used to obtain a degree $s = 2k-1$ Gaussian quadrature rule or a degree $2k$ quadrature by interpolation or approximation rule.
Random vectors are drawn from the uniform distribution on the hypersphere.

\subsection{Comparison with classical quadrature}
\label{sec:classical}

We begin with some experiments meant to highlight the differences between classical quadrature and the quadrature algorithms studied in this paper.
The main takeaway of these experiments is that whether bounds based on polynomial approximation on a single interval $[\lmin,\lmax]$ are actually representative of the true behavior of \cref{alg:protoalg} depends on properties of the spectrum of $\vec{A}$ (and $f(x)$). 
In particular, algorithms based on Gaussian quadrature often exhibit convergence much faster than predicted by such bounds when the spectrum of $\vec{A}$ has gaps or outlying eigenvalues.

In addition, these experiments illustrate similarities and differences between the behavior of classical quadrature rules for continuous weight functions and the behavior of the matrix-free quadrature algorithms presented in this paper. 
First, the costs of the algorithms in this paper are determined primarily by the number of matrix-vector products.
This is because we typically only want to approximate \( \int f\d\Psi \) for a single, or perhaps a few, functions. 
On the other hand, the weight functions which classical quadrature rules approximate never change, so nodes and weights can be precomputed and the dominant cost becomes the cost to evaluate \( f \) at the quadrature nodes.
Second, while classical weight functions, such as the weight functions for Jacobi or Hermite polynomials, are typically relatively uniform in the interior of the interval of integration the weighted CESM \( \Psi \) may vary wildly from application to application. 
In some cases \( \Psi \) may resemble the distribution function of a classical weight function whereas in others it might have large gaps, jumps, and other oddities.

\begin{figure}
    \includegraphics[width=\textwidth]{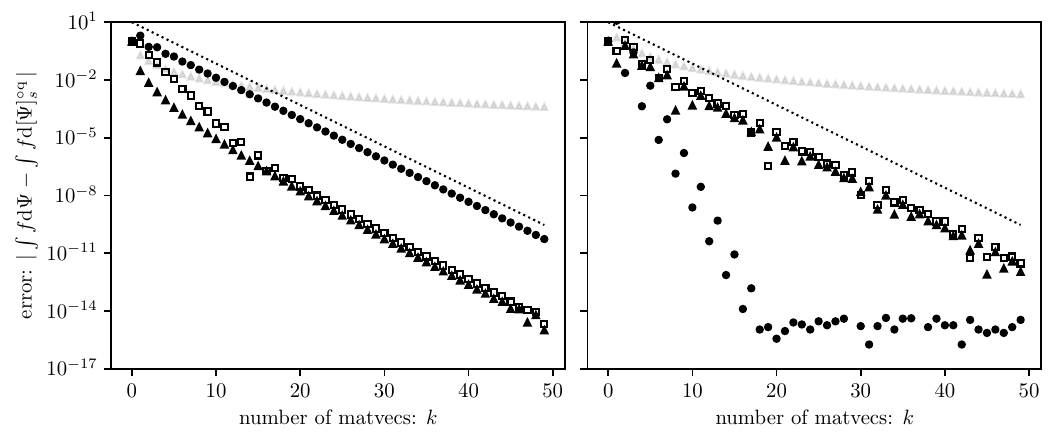}
    \caption{
    Errors for approximating \( \int f \d\Psi =  \vec{v}^\cT f(\vec{A}) \vec{v} \) when \( f(x) = 1/(1+16x^2) \) for a spectrum uniformly filling \( [-1,1] \) (left) and a spectrum with a gap around zero (right).
    \emph{Legend}: Gaussian quadrature 
    ({\protect\raisebox{0mm}{\protect\includegraphics[scale=.7]{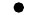}}}), 
    quadrature by interpolation 
    ({\protect\raisebox{0mm}{\protect\includegraphics[scale=.7]{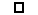}}}),
    quadrature by approximation 
    ({\protect\raisebox{0mm}{\protect\includegraphics[scale=.7]{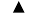}}}),
    Jackson's damped quadrature by approximation 
    ({\protect\raisebox{0mm}{\protect\includegraphics[scale=.7]{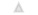}}}), and rate $O(\rho^{-2k})$, where $\rho = (1+\sqrt{17})/4$.
    The behavior of the algorithms are highly dependent on the eigenvalue distribution of \( \vec{A} \), but intuition about classical approximation theory informs our understanding.
    Gaussian quadrature may perform significantly better that explicit methods when the spectrum of \( \vec{A} \) has additional structure such as gaps. 
    }
    \label{fig:GQ_CC_tre08}
\end{figure}

Throughout this example, we use the Runge function \( f(x) = 1/(1+16x^2) \) and a vector \( \vec{v} \) with uniform weight on each component.
We will compare the effectiveness of the Gaussian quadrature rule \( \qq[g]{\Psi}{2k-1} \), the quadrature by interpolation rule \( \qq[i]{\Psi}{2k} \) and the quadrature by approximation rule \( \qq[a]{\Psi}{2k} \). 
For the latter approximations, we set \( \mu = \mu_{-1,1}^T \), and for the quadrature by approximation rule, we    approximate the involved integrals to essentially machine precision.
All three approximations can be computed using \( k \) matrix-vector products with \( \vec{A} \), and since the approaches exactly integrate polynomials of degree \( 2k-1 \) and \( 2k \) respectively, we might expect them to behave similarly.
However, there are a variety of factors which prevent this from being the case.

In our first experiment, shown on the left panel of \cref{fig:GQ_CC_tre08}, the spectrum of \( \vec{A} \) uniformly fills out the interval \( [-1,1] \); i.e., \( \lambda_i = -1+(2i+1)/n \), \( i=0,1,\ldots, n-1 \).
We take \( n=10^5 \) so that \( \qq[g]{\Psi}{2k-1} \) and \( \qq[i]{\Psi}{2k} \) respectively approximate the \( k \)-point Gaussian quadrature and \( (2k-1) \)-point Fej\'{e}r quadrature rules for a uniform weight on \( [-1,1] \).
For many functions, certain quadrature by interpolation rules on \( [-1,1] \), including the Fej\'{e}r rule, behave similarly to the Gaussian quadrature rule when the same number of nodes are used \cite{trefethen_08}.
For \( f(x) = 1/(1+16x^2) \), this phenomenon is observed for some time until the convergence rate is abruptly cut in half \cite{weideman_trefethen_07}.
In our setting, a fair comparison means that the number of matrix-vector products are equal, and so we see, in the left panel of \cref{fig:GQ_CC_tre08}, that the quadrature by interpolation approximation initially converges twice as quickly as the Gaussian quadrature approximation!
The rate of the quadrature by interpolation approximation is eventually cut in half to match the rate of the Gaussian quadrature approximation, $O(\rho^{-2k})$, where $\rho = (1+\sqrt{17})/4$ \cite{weideman_trefethen_07}.

In our second experiment, shown on the right panel of \cref{fig:GQ_CC_tre08}, the spectrum of \( \vec{A} \) uniformly fills out the disjoint intervals \( [-1,-.75] \cup [.75,1] \) with the same inter-point spacing as the first example; i.e. we remove the eigenvalues in the previous example which fall between \( -.75 \) and \( .75 \).
Here we observe that the Gaussian quadrature rule converges significantly faster than in the previous experiment.
This is to be expected. 
Indeed, the Gaussian quadrature rule has its nodes near \( [-1,-.75]\cup[.75,1] \), so the union of the support of \( \Psi \) and \( \qq[i]{\Psi}{2k} \) is further from the poles of \( f \) located at \( \pm \im /4 \).
We also note that the conditions which enabled accelerated convergence in the first experiment are no longer present, so the quadrature by interpolation approximation converges at its limiting rate \cite{trefethen_08}.

\subsection{Approximating sparse spectra}

If the spectrum of \( \vec{A} \) is \( S \)-sparse; i.e., there are only \( S \) distinct eigenvalues, then the \( k \)-point Gaussian quadrature rule will be exactly equal to the weighted CESM for all \( k \geq S \), at least in exact arithmetic.
Thus, the runtime required by SLQ is determined by \( S \) and the number of samples of the weighted CESM which are required to get a good approximation to the true CESM.
The interpolation and approximation based approaches, which are based on the orthogonal polynomials of some fixed distribution function \( \mu \), are unable to take advantage of such sparsity.
Indeed, unless the eigenvalues of \( \vec{A} \) are known a priori, such methods have fixed resolution due to the fixed locations of the zeros of the orthogonal polynomials with respect to \( \mu \).
Moreover, quadrature by approximation methods suffer from Gibbs oscillations unless a damping kernel is used, in which case the resolution is further decreased.

\begin{figure}
    \includegraphics[width=\textwidth]{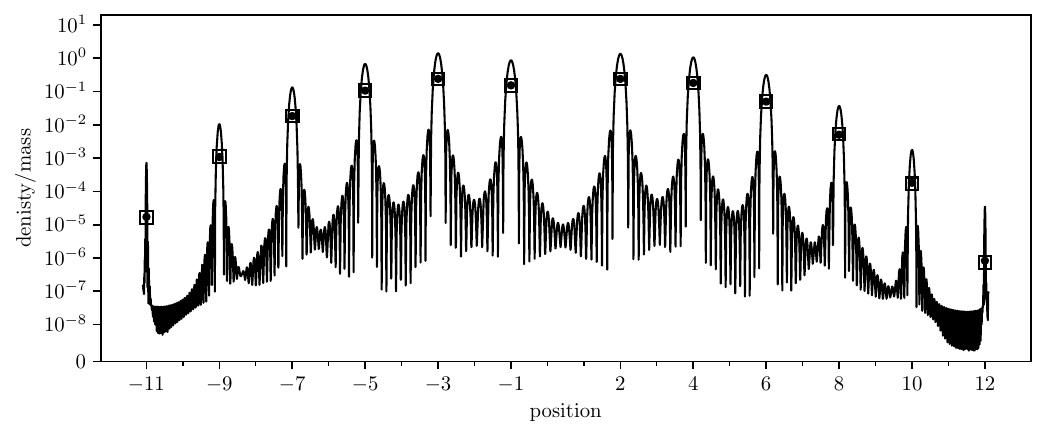}
    \caption{
    Approximations to a sparse spectrum with just \( 12 \) eigenvalues.
    \emph{Legend}:
    true spectral density  
    ({\protect\raisebox{0mm}{\protect\includegraphics[scale=.7]{imgs/legend/square_empty.pdf}}}), 
    Gaussian quadrature approximation: \( s=23 \)
    ({\protect\raisebox{0mm}{\protect\includegraphics[scale=.7]{imgs/legend/circle.pdf}}}), 
    damped quadrature by approximation: \( s=500 \)
    ({\protect\raisebox{0mm}{\protect\includegraphics[scale=.7]{imgs/legend/solid_thin.pdf}}}).
    The Gaussian quadrature produces an extremely good approximation using just \( 12 \) matrix-vector products.
    Even with many more matrix-vector products, damped quadrature by approximation does not have the same resolution. 
   }
    \label{fig:Kneser}
\end{figure}

\begin{figure}
    \includegraphics[width=\textwidth]{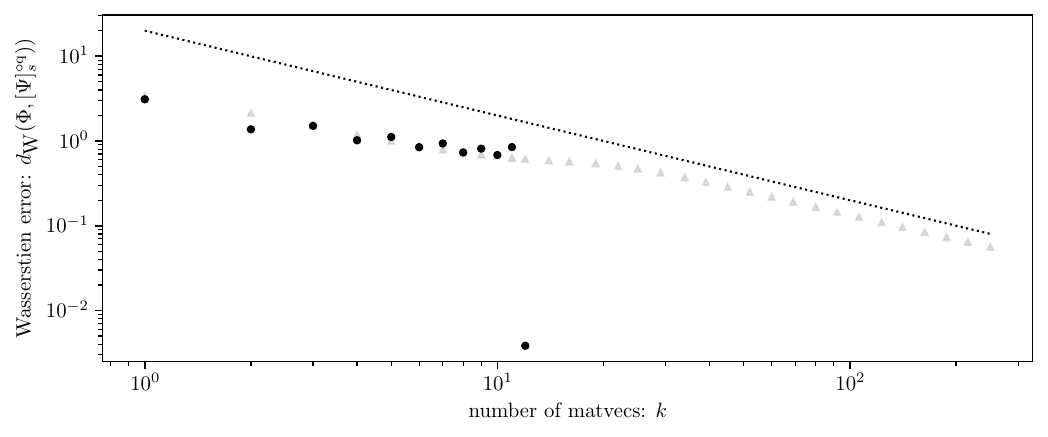}
    \caption{
    Wasserstien error for approximating CESM \( \Phi \) supported on just 12 points.
    \emph{Legend}: Gaussian quadrature 
    ({\protect\raisebox{0mm}{\protect\includegraphics[scale=.7]{imgs/legend/circle.pdf}}}), 
    Jackson's damped quadrature by approximation 
    ({\protect\raisebox{0mm}{\protect\includegraphics[scale=.7]{imgs/legend/tria.pdf}}}),
    and reference for rate $O(s^{-1})$    ({\protect\includegraphics[scale=.7]{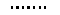}}).
    The Gaussian quadrature rule converges very quickly (in just 12 steps) to the true weigthed CESM, while the damped quadrature by approximation converges at a rate of $O(s^{-1})$.
    }
    \label{fig:Kneser_convergence}
\end{figure}

In this example, we approximate the CESM of the adjacency matrix of a Kneser graph.
The \( (N,K) \)-Kneser graph is the graph whose vertices correspond to size \( K \) subsets of \( \{1, 2, \ldots, N\} \) and whose edges connect vertices corresponding to disjoint sets.
It is not hard to see that the number of vertices is \( \binom{N}{K} \) and the number of edges is \( \frac{1}{2}\binom{N}{K}\binom{N-K}{K} \).
The spectrum of Kneser graphs is known as well. 
Specifically, there are \( K+1 \) distinct eigenvalues whose values and multiplicities are:
\begin{equation}
    \lambda_i = (-1)^i\binom{N-K-i}{K-i}
    ,\qquad
    m_i = \binom{N}{i} - \binom{N}{i-1}
    ,\qquad 
    i=0,1,\ldots, K,
\end{equation}
where $m_0$ is defiend to be 1.

We conduct a numerical experiment with \( N = 23 \) and \( K = 11 \), the same values used in \cite{large_matrices_density_review_18}.
This results in a graph with 1,352,078 vertices and 8,112,468 edges.
Thus, the adjacency matrix is highly sparse. 
We compare the Gaussian quadrature approximation with the damped quadrature by approximation. 
In both cases, we use a single random test vector \( \vec{v} \).
For the Gaussian quadrature, we set \( k = 12 \).
For the damped quadrature by approximation we set \( s=500 \) and use Jackson damping with \( \mu = \mu_{a,b}^T \), where \( a = -11.1 \) and \( b=12.1 \).

In \cref{fig:Kneser} we show what the approximations look like for fixed values of $s$.
Note that at $s=23$the Gaussian quadrature matches almost exactly despite having used only \( k=12 \) matrix-vector products, although the estimated weights of low-multiplicity eigenvalues have higher relative error.
On the other hand, even after \( s=500 \) matrix-vector products, the damped quadrature by approximation has a much lower resolution.

In \cref{fig:Kneser_convergence} we show the Wasserstien distance between the true CESM $\Psi$ and the Gaussian quadrature and damped quadrature by approximation rules at varying values of $s$. 
While the Gaussian quadrature is exact when $s = 23$ ($k=12$), the damped quadrature by approximation converges at a rate $O(s^{-1})$, matching the rate of the upper bound \cref{thm:spec_approx}.

\begin{remark}
There are sublinear time algorithms for approximate matrix-vector products with the (normalized) adjacency matrix.
Specifically, in a computational model where it is possible to 
(i) uniformly sample a random vertex in constant time, 
(ii) uniformly sample a neighbor of a vertex in constant time, and 
(iii) read off all neighbors of a vertex in linear time, 
then an \( \varepsilon_{\textup{mv}} \)-accurate approximate to the matrix-vector product with the adjacency matrix can be computed, with probability \( 1-\eta \), in time \( O(n (\varepsilon_{\textup{mv}})^{-2}\ln(\eta^{-1})) \).
For dense graphs, this is sublinear in the input size \( O(n^2) \) of the adjacency matrix.
See \cite{braverman_krishnan_musco_22,jin_karmarkar_musco_sidford_singh_24} for an analysis in the context of spectrum approximation.
\end{remark}

\subsection{Approximating ``smooth'' densities}

There are a range of settings in which the spectral density of \( \vec{A} \) is close to a smooth slowly varying density.
In such cases, we might hope that our approximation satisfies certain known criteria. 
For instance, that the approximation is also a slowly varying density, that the behavior of the approximation at the endpoints of the support satisfies the right growth or decay conditions, etc.
In this example, we consider how parameters in \cref{alg:protoalg} can be varied so that the resulting approximation enjoys certain desired properties.

One setting in which \( \vec{A} \) may have a slowly varying density is when \( \vec{A} \) is a large random matrix. 
We begin this example by considering a sample covariance matrix
\begin{equation}
    \vec{A}_n = \frac{1}{\nv} \vec{\Sigma}^{1/2} \vec{X} \vec{X}^\cT \vec{\Sigma}^{1/2}
\end{equation}
where \( \vec{X} \) is random and \( \vec{\Sigma} \) is deterministic. 
Specifically, we fix constants \( \sigma > 1 \) and \( d \in (0,1) \), define \( m = n/d \), and take \( \vec{X} \) to be a \( n\times m \) matrix with iid standard normal entries and \( \vec{\Sigma} \) a diagonal matrix with \( 1/m \) as the first \( n/2 \) entries and \( \sigma/m \) as the last \( n/2 \) entries.

In the limit, as \( n\to\infty \), the CSEM \( \Phi_n \) of \( \vec{A}_n \) is convergent to distribution \( \Psi_\infty \) supported on two disjoint intervals \( [a_1,b_1] \cup [a_2,b_2] \), where \( a_1 < b_1 < a_2 < b_2 \), with equal mass on each \cite{bai_silverstein_98}.
The spectral edges are equal to the values at which 
\begin{equation}
    x\mapsto  -\frac{1}{x} + \frac{d}{2} \left( \frac{1}{x+1} + \frac{1}{x+\sigma^{-1}} \right)
\end{equation}
attains at its local extrema.
Moreover, it is known that \( \d\Psi_\infty/\d x \) has square root behavior at the spectral edges.

Because we know the support of the desired density, and because we know the behavior at the spectral edges, a natural choice is to use quadrature by approximation with 
\begin{equation}
    \mu(x) = \tfrac{1}{2}\mu_{a_1,b_1}^U(x) + \tfrac{1}{2} \mu_{a_2,b_2}^U(x)
\end{equation}
where \( \mu_{a,b}^U \) is the weight function for the Chebyshev polynomials of the second kind given by
\begin{equation}
    \mu_{a,b}^U(x) 
    = \int_{a}^{x} \frac{4}{\pi (b-a)} \sqrt{1-\left( \frac{2}{b-a}t - \frac{b+a}{b-a}\right)} \d t.
\end{equation}
This will ensure that the Radon--Nikodym derivative \( \d\Psi_\infty / \d\mu \) is of order 1 at the spectral edges which seems to result in better numerical behavior than if we were to use a KPM approximation corresponding to a density which explodes at the spectral edges.

To compute the Jacobi matrix for \( \mu \), make sure of the fact that the product
\begin{equation}
    \int p \d\mu
    = \frac{1}{2} \int p \d\mu_{a_1,b_1}^U
    + \frac{1}{2} \int p \d\mu_{a_2,b_2}^U
\end{equation}
can be computed \emph{exactly} by applying a sufficiently high degree quadrature rule to each of the right hand side integrals. 
We can then apply the Lanczos algorithm to compute the orthogonal polynomials.
Some potentially more computationally efficient approaches are outlined in \cite{fischer_golub_91}.

\begin{figure}
    \includegraphics[width=\textwidth]{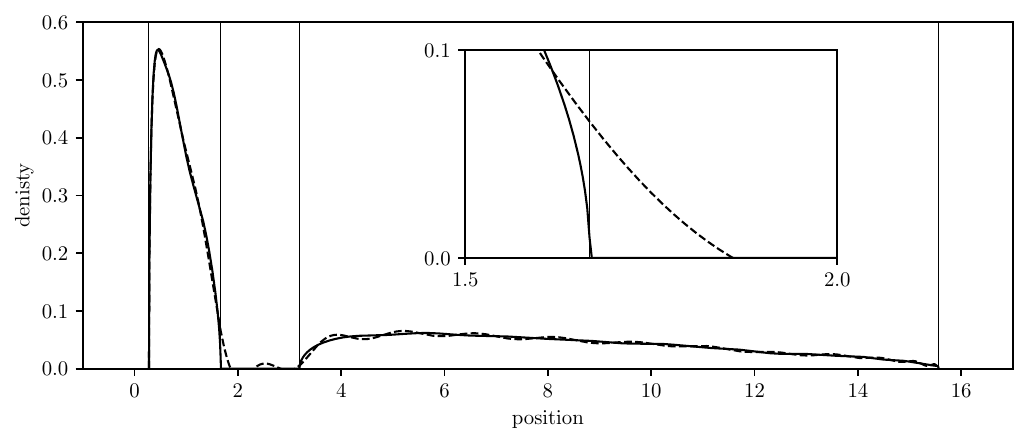}
    \caption{
        Approximations to a ``smooth'' spectrum using quadrature by approximation with various choices of \( \mu \).
        \emph{Legend}:
        \( \mu = \mu_{a_1,b_2}^U \) ({\protect\raisebox{0mm}{\protect\includegraphics[scale=.7]{imgs/legend/dash_thin.pdf}}}),
        \( \mu = \frac{1}{2} \mu_{a_1,b_1}^U + \frac{1}{2} \mu_{a_2,b_2}^U \)({\protect\raisebox{0mm}{\protect\includegraphics[scale=.7]{imgs/legend/solid_thin.pdf}}}).
        A priori knowledge about properties of the spectrum allow better choices of parameters such as \( \mu \). 
    }
    \label{fig:RM_split_AQ}
\end{figure}

We conduct a numerical experiment with \( n = 10^4 \), \( d=0.3 \), and $\sigma = 8$.
We use \( k=15 \) and average over 10 trials, resampling \( \vec{A}_n \) in each trial.
To generate an approximation to the density we expand the support of the limiting density by \( 0.001 \) at each endpoint to avoid eigenvalues of \( \vec{A}_n \) lying outside the support of \( \mu \).
In \cref{fig:RM_split_AQ} we show the approximations with \( \mu = \mu_{a_1,b_2}^U \) and \( \mu = \frac{1}{2} \mu_{a_1,b_1}^U + \frac{1}{2} \mu_{a_2,b_2}^U \).
As shown in the inset image of \cref{fig:RM_split_AQ}, we observe that the approximation with \( \mu = \frac{1}{2} \mu_{a_1,b_1}^U + \frac{1}{2} \mu_{a_2,b_2}^U \) exhibits the correct square root behavior at the endpoints as well as fewer oscillations throughout the interior of the support of the density.

\begin{remark}
In recent work \cite{ding_trogdon_21} it was shown how Lanczos performs on such a sample covariance matrix. 
In particular, one sample from Stochastic Lanczos Quadrature will converge almost surely, as \( n\to\infty \), to the desired distribution.
In this same work another density approximation scheme was proposed based on Stieltjes transform inversion. 
Analysis and comparison for this method is an interesting open problem.
\end{remark}

\subsubsection{Smoothing by convolution}
\label{sec:smoothing}

Quadrature by interpolation produces a sum of weighted Dirac delta functions.
A simple approach to obtain a density function from a distribution function involving point masses is to approximate each point masses with some concentrated probability density function; e.g. Gaussians with a small variance \cite{lin_saad_yang_16,ghorbani_krishnan_xiao_19}.
This is simply convolution with this distribution, and if the smoothing distribution has small enough variance, the Wasserstein distance between the original and smoothed distributions will be small. 
 Specifically, we have the following standard lemma which we prove in \cref{sec:smoothing_wass} of the supplementary materials:
\begin{lemma}
    \label{thm:smoothing_wass}
    Given a smooth positive probability distribution function \( G_\sigma \), define the smoothed approximation \( \Upsilon_\sigma \) to \( \Upsilon \) by the convolution
    \begin{equation}
        \Upsilon_\sigma(x)  = \int_{-\infty}^{\infty} G_\sigma(x-t) \d\Upsilon(t).
    \end{equation}
    Then, \( \W(\Upsilon,\Upsilon_\sigma) \leq \W(\bOne(\:\cdot < 0), G_\sigma) \).
    Moreover, if \( G_\sigma \) has median zero and standard deviation \( \sigma \), then \( \W(\Upsilon,\Upsilon_\sigma)  \leq \sigma \). 
\end{lemma}

\begin{figure}
\includegraphics[width=\textwidth]{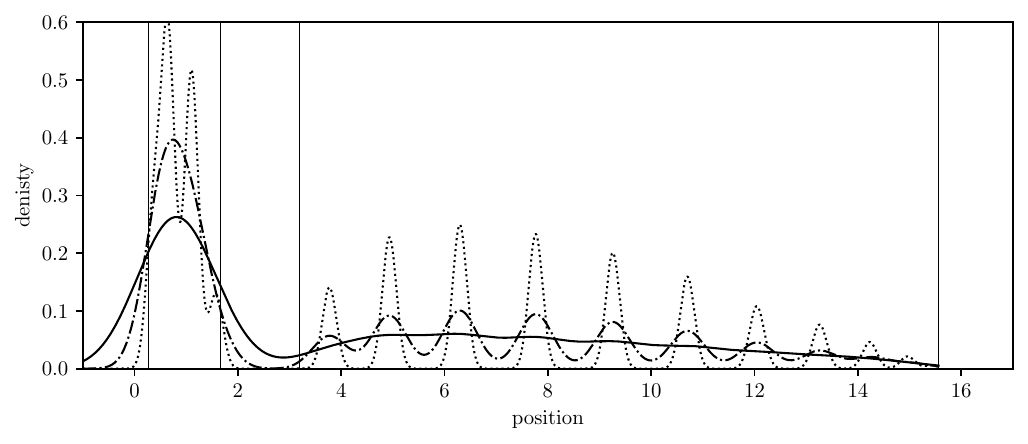}
\caption{
    Approximations to a ``smooth'' spectrum using smoothed Gaussian quadrature for various smoothing parameters \( \sigma \).
    The true limiting density is approximately supported on $[0.28,1.67]\cup[ 3.19,15.6]$.
    \emph{Legend}:
    \( \sigma = 2/k \) ({\protect\raisebox{0mm}{\protect\includegraphics[scale=.7]{imgs/legend/dot_thin.pdf}}}),
    \( \sigma = 5/k \) ({\protect\raisebox{0mm}{\protect\includegraphics[scale=.7]{imgs/legend/dash_thin.pdf}}}),
    \( \sigma = 10/k \) ({\protect\raisebox{0mm}{\protect\includegraphics[scale=.7]{imgs/legend/solid_thin.pdf}}}).
    Obtaining a density from a discrete distribution via convolutions isn't always a good idea.
}
\label{fig:GQ_conv_smoothing}
\end{figure}

It is well known that if \( G_\sigma \) is differentiable, then the smoothed distribution function \( \Upsilon_\sigma \) will also be differentiable.
Thus, we can obtain a density function \( \d\Upsilon_\sigma/\d{x} \) even if \( \Upsilon \) has discontinuities.
Moreover, the bounds obtained earlier can easily be extended to smoothed spectral density approximations obtained by convolution using the triangle inequality.

While the smoothing based approach has a simple theoretical guarantee in Wassetherstein distance, it does not need to provide a good approximation to the density. 
Indeed, if the variance of the smoothing kernel is too small, then the smoothed distribution will still look somewhat discrete. 
On the other hand, if the variance of the smoothing kernel is too large, then the smooth distribution will become blurred out and lose resolution.
As shown in \cref{fig:GQ_conv_smoothing}, this is particularly problematic if different parts of the spectrum would naturally require different amounts of smoothing.

There are of course many different smoothing schemes that could be used. 
These include adaptively choosing the variance parameter based on the position in the spectrum, using a piecewise constant approximation to the density, interpolating the distribution function with a low degree polynomial or splines, etc. 
Further exploration of these approaches is beyond the scope of this paper since they would likely be context dependent. 
For instance, in random matrix theory, it may be desirable to enforce square root behavior at endpoints whereas in other applications it may be desirable to have smooth tails.
More broadly, as illustrated in several of the numerical experiments above, approaches based on quadrature by approximation, with carefully chosen $\mu$, seem to have good potential.

We conclude with the remark that alternate metrics of closeness, such as the total variation distance, are likely better suited for measuring the quality of approximations to ``smooth'' densities.
However, since the actual spectral density \( \d\Psi/\d{x} \) is itself the sum of Dirac deltas, some sort of regularization is required to obtain a proper density \cite{lin_saad_yang_16} which of course relates closely to what it actually means to be ``close to a smooth slowly varying density''.
A rigorous exploration of this topic would be of interest.

\subsubsection{Handling spikes}

In some situations, one may encounter spectra which are nearly ``smooth'' except at a few points where there are large jumps in the CESM. 
For instance, low rank matrices may have many repeated zero eigenvalues. 

To model such a situation, we consider a matrix 
\begin{equation}
    \vec{A}_n = 
    \begin{bmatrix}
        m^{-1} \vec{X}\vec{X}^\cT & \vec{0} \\
        \vec{0} & z\vec{I} + \sigma \vec{D}
    \end{bmatrix}
\end{equation}
where \( \vec{X} \) is a \( n'\times m \) matrix with standard normal entries and \( \vec{D} \) is a \( (n-n')\times (n-n') \) diagonal matrix with standard normal entries. 
In both cases,  \( m = n'/d \) for some fixed \( d\in(0,1) \). 
While this particular matrix is block diagonal, the protoalgorithm is mostly oblivious to this structure and would work similarly well if the matrix were conjugated by an arbitrary unitary matrix so that the block diagonal structure were lost.

When \( n\to\infty \) and \( \sigma \to 0 \), the spectral density \( \d\Phi_n/\d{x} \) is convergent to a density \( \d\Phi_\infty/\d{x} \) equal to the sum of a scaled Marchenko--Pastur distribution and a weighted Dirac delta distribution. 
Thus, a natural approach would be to use quadrature by approximation with $\mu$ defined by
\begin{equation}
    \mu(x) = (1-p) \mu_{a,b}^{U}(x) + p \delta(x-z).
\end{equation}
As above, we can use a modified version of the Vandermonde with Arnoldi approach to compute the orthogonal polynomials with respect to \( \mu \). 
The resulting approximation to the ``smooth'' part of the density \( \d\Phi/\d{x} \) is shown in \cref{fig:RM_spike}.

\begin{figure}
\includegraphics[width=\textwidth]{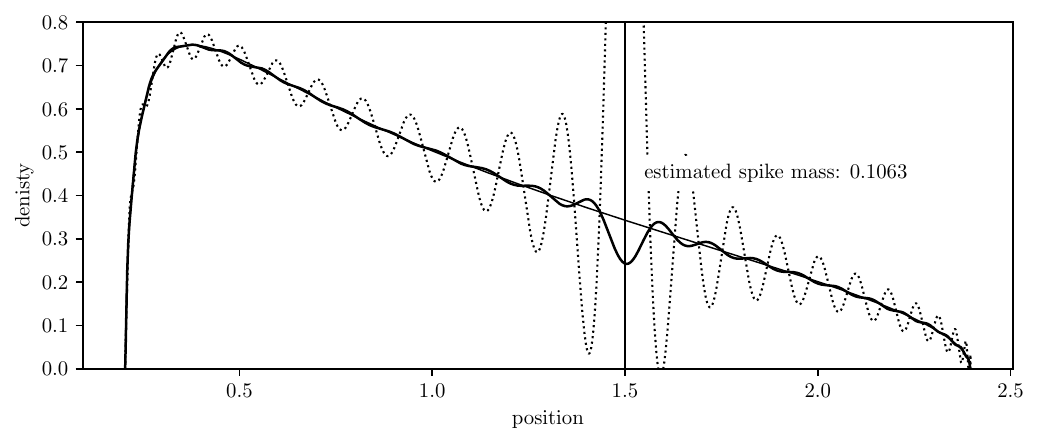}
\caption{
    Approximations to a ``smooth'' spectrum with a spike using quadrature by approximation with various choices of \( \mu \).
    \emph{Legend}:
    absolutely continuous part of true limiting density ({\protect\raisebox{0mm}{\protect\includegraphics[scale=.7]{imgs/legend/solid_thin.pdf}}}),
    quadrature by approximation: \( \mu = (1-p) \mu_{a,b}^{U} + p \delta(\:\cdot -z) \) ({\protect\raisebox{0mm}{\protect\includegraphics[scale=.7]{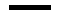}}}),
    quadrature by approximation: \( \mu = \mu_{a,b}^U \)({\protect\raisebox{0mm}{\protect\includegraphics[scale=.7]{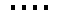}}}).
    A priori knowledge of the location of a singularity allows for a better approximation to the absolutely continuous part of the spectrum. 
    }
\label{fig:RM_spike}
\end{figure}

We set \( n = 10^6 \), \( n' = n/10 \), \( d = 0.3 \), \( z=1.5 \), and \( \sigma = 10^{-10} \).
As before, we average of \( 10 \) trials where \( \vec{A}_n \) is resampled in each trial.
For each sample, we compute the quadrature by approximation with \( k=25 \) for \( \mu = (1-p) \mu_{a,b}^{U} + p \delta(x-z) \) with \( p = 0.2 \) and \( \mu = \mu_{a,b}^U \).
We approximate $\delta(x-z)$ by $\mu_{z-\epsilon,z+\epsilon}^U$ for $\epsilon = 0.001$, and take $a,b$ as the support of the limiting Marchecnko--Pastur distribution expanded by $0.001$.
The results are show in \cref{fig:RM_spike}.

Clearly, accounting for the spike explicitly results in a far better approximation to the density.
Note that this approach does not require that the mass of the spike is accurately matched in the reference distribution $\mu$. 
For instance, in our example, we set $p$ to be 0.2 while the actual mass is 0.1. Even so, the approximation is able to correct this and we still recover essentially the correct mass.
On the other hand, if the location of the spike is incorrectly estimated, then the approximation to the density may have massive oscillations.
In our example the spike has width roughly \( 10^{-10} \) which does not cause issues for the value of \( s \) used. 
However, if \( s \) is increased, the width of the spike is increased, or the location of the estimate of the spike is offset significantly, then the existing oscillations become large.
Approaches for adaptively finding the location of spikes would be an interesting area of further study; see \cite{chen_23} for an adaptive variant of the KPM which works towards this goal.

\subsection{Thermodyanmics of spin systems}
\label{sec:spin}

The quantum Heisenberg model can be used to study observables of magnetic systems \cite{weisse_wellein_alvermann_fehske_06,schnalle_schnack_10,schnack_richter_steinigeweg_20,schulter_gayk_schmidt_honecker_schnack_21}.
For a system with \( N \) spins of spin number \( S \), the Heisenberg spin Hamiltonian is an operator on a Hilbert space of dimension \( (2S+1)^N \) given by
\begin{equation}
    \vec{H} = \sum_{i=0}^{N-1} \sum_{j=0}^{N-1} \left( 
     [\vec{J}^{\textup{x}}]_{i,j} \vec{s}^{\textup{x}}_i  \vec{s}^{\textup{x}}_j 
    +[\vec{J}^{\textup{y}}]_{i,j}  \vec{s}^{\textup{y}}_i  \vec{s}^{\textup{y}}_j
    +[\vec{J}^{\textup{z}}]_{i,j} \vec{s}^{\textup{z}}_i  \vec{s}^{\textup{z}}_j
    \right).
\end{equation}
Here \( \vec{s}^\sigma_i \) gives the component spin operator for the \( i \)-th spin site and acts trivially on the Hilbert spaces associated with other spin sites but as the \( (2S+1)\times (2S+1) \) component spin matrix \( \vec{s}^{\sigma} \) on the \( i \)-th spin site.
Thus, \( \vec{s}^\sigma_i \) can be represented in matrix form as
\begin{equation}
    \vec{s}^\sigma_i
    = \underbrace{\vec{I} \otimes \cdots \otimes \vec{I}}_{i\text{ terms}} 
    \otimes ~ \vec{s}^\sigma \otimes 
    \underbrace{\vec{I} \otimes \cdots \otimes \vec{I}}_{N-i-1\text{ terms}}.
\end{equation}

\begin{figure}
    \includegraphics[width=\textwidth]{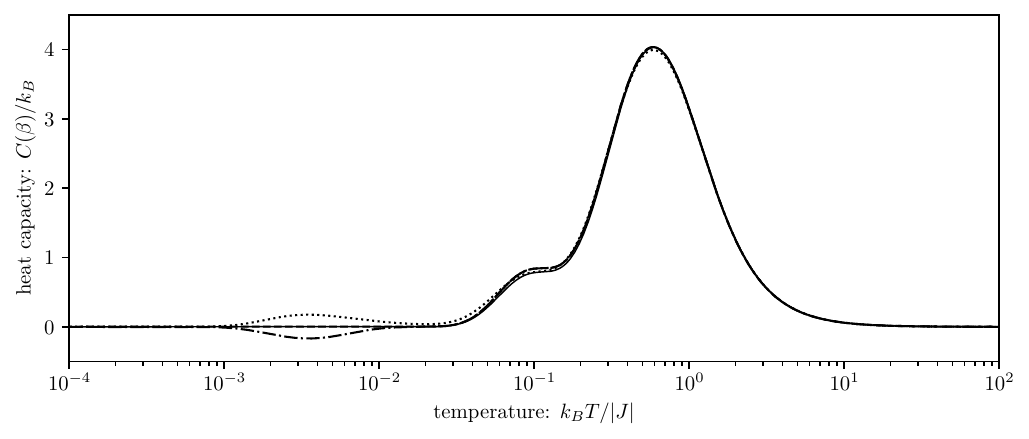}
    \caption{
    Heat capacity as a function of temperature for a small spin system.
    \emph{Legend}:
    exact diagonalization ({\protect\raisebox{0mm}{\protect\includegraphics[scale=.7]{imgs/legend/solid_thin.pdf}}}),
    Gaussian quadrature
    ({\protect\raisebox{0mm}{\protect\includegraphics[scale=.7]{imgs/legend/dash_thin.pdf}}}),
    quadrature by approximation 
    ({\protect\raisebox{0mm}{\protect\includegraphics[scale=.7]{imgs/legend/dashdot_thin.pdf}}}),
    damped quadrature by approximation 
    ({\protect\raisebox{0mm}{\protect\includegraphics[scale=.7]{imgs/legend/dot_thin.pdf}}}).
    While damping produces a physical result, the resulting ghost bump may be more difficult to identify than the nonphysical ghost dip obtained without damping.
    }
    \label{fig:spin_heat_capacity}
\end{figure}

The CESM of \( \vec{H} \) gives the energy spectrum of the system and can be used to compute many important quantities.
For instance, given an observable \( \vec{O} \) (i.e. a Hermitian matrix), the corresponding thermodynamic expectation of the observable is given by
\begin{equation}
    \frac{\tr(\vec{O}\exp(-\beta \vec{H}))}{\tr(\exp(-\beta\vec{H}))}.
\end{equation}

Quantities depending on observables which are matrix functions \( \vec{H} \) can be written entirely in terms of matrix functions of \( \vec{H} \).
For instance, with \( \beta = (k_B T)^{-1} \), where \( k_B \) is the Boltzmann constant and \( T \) is the temperature, the system heat capacity is given by
\begin{equation}
    \frac{C(T)}{k_B}
    = 
    \frac{\tr\left( (\beta\vec{H})^2 \exp(-\beta \vec{H}) \right)}{\tr \left( \exp(-\beta \vec{H}) \right)}
    -\left[ \frac{\tr\left( \beta \vec{H} \exp(-\beta \vec{H}) \right)}{\tr \left( \exp(-\beta \vec{H}) \right)}\right]^2.
\end{equation}
Thus, for fixed finite temperature, evaluating the heat capacity amounts to evaluating several matrix functions. 

\begin{figure}
    \includegraphics[width=\textwidth]{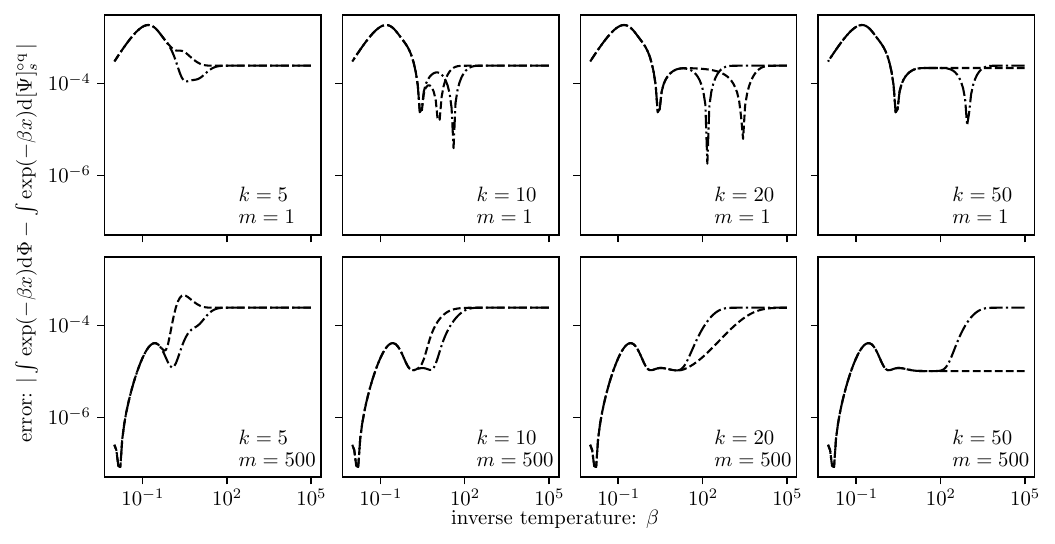}
    \caption{
    Error of partition function $\tr(\exp(-\beta \vec{A}))$ as a function of inverse temperature $\beta$ for a small spin system for various choices of $k$ and $m$.
    \emph{Legend}:
    Gaussian quadrature
    ({\protect\raisebox{0mm}{\protect\includegraphics[scale=.7]{imgs/legend/dash_thin.pdf}}}),
    quadrature by approximation 
    ({\protect\raisebox{0mm}{\protect\includegraphics[scale=.7]{imgs/legend/dashdot_thin.pdf}}}).
    As guaranteed by \cref{thm:unif_anl} (see also \cref{thm:exp_unif}), the Gaussian quadrature method is simultaneously accurate for an infinite number different functions. 
    This also appears to be the case for the quadrature by approximation method, although to a lesser extent. 
    }
    \label{fig:spin_partition}
\end{figure}

In some cases, symmetries of the system can be exploited to diagonalize or block diagonalize \( \vec{H} \) \cite{schnalle_schnack_10}. 
Numerical diagonalization can be applied to blocks to obtain a full diagonalization.
Even so, the exponential dependence of the size of \( \vec{H} \) on the number of spin sites \( N \) limits the size of systems which can be treated in this way. 
Moreover, such techniques are not applicable to all systems. 
Thus, approaches based on \cref{alg:protoalg} are widely used; see \cite{schnack_richter_steinigeweg_20} for examples using a Lanczos based approach and \cite{schulter_gayk_schmidt_honecker_schnack_21} for examples using a Chebyshev based approach.

In this example, we consider a Heisenberg ring (\( [\vec{J}^{\textup{x}}]_{i,j} = [\vec{J}^{\textup{y}}]_{i,j} = [\vec{J}^{\textup{z}}]_{i,j} = \bOne(|i-j| = 1 \mod{N}) \)) with \( N = 12 \) and \( S = 1/2 \).
Similar examples, with further discussion in the context of the underlying physics, are considered in \cite{schnack_richter_steinigeweg_20,schulter_gayk_schmidt_honecker_schnack_21}.
We take \( k = 50 \) and \( \nv = 300 \) and compute approximations to the heat capacity at many temperatures using Gaussian quadrature, quadrature by interpolation, and damped quadrature by interpolation. 
For the latter two approximations we use \( \mu = \mu_{a,b}^T \) where \( a \) and \( b \) are chosen based on the nodes of the Gaussian quadrature. 
Note that averages over random vectors are computed for each trace rather than for \( C(\beta) \), and that we use the same vectors for all four traces. 
The results are shown in \cref{fig:spin_heat_capacity}.

We note the presence of a nonphysical ``ghost dip'' in the quadrature by approximation. 
If the approximation to the CESM is non-decreasing, the Cauchy--Schwarz inequality guarantees a positive heat capacity.
Thus, when we use Jackson's damping, the heat capacity remains positive for all temperatures.
However, as noted in \cite{schulter_gayk_schmidt_honecker_schnack_21}, this is not necessarily desirable as the ghost dip is easily identifiable while the ghost peak may be harder to identify.
We remark that it would be interesting to provide a posteriori bounds for the accuracy of the approximations to the quantity \( \tr(\vec{O} \exp(-\beta \vec{H})) / \tr(\exp(-\beta \vec{H}))  \).

Another important quantity is the partition function
\[
Z(\beta) := \tr(\exp(-\beta \vec{H})),
\]
which encodes many thermodynamic quantities \cite{schnack_richter_steinigeweg_20,schulter_gayk_schmidt_honecker_schnack_21}.
In \cref{fig:spin_partition} we show the error of \cref{alg:protoalg} for $\circ = \mathrm{g}$ and $\circ = \mathrm{a}$ with $\mu = \mu_{\lmin,\lmax}^T$.
In particular, we plot the error for $\circ = \mathrm{g}$ over a range of $\beta$ for various choices of $s$ and $m$.
Here we normalize the spectrum of $\vec{H}$ so that $\lmin=0$ by adding a multiple of the identity. 
This amounts to scaling the partition functions by a constant. 

As guaranteed by \cref{thm:unif_anl} (see also \cref{thm:exp_unif}), the Gaussian quadrature method is simultaneously accurate for an infinite number of different functions.
We are unaware of any past work which could guarantee this.
As expected, increasing $m$ and $s$ decreases the error. 
While we do not have theoretical guarantees, we make a similar qualitative observation for $\circ = \mathrm{a}$, although the error seems to be larger than $\circ = \mathrm{g}$, especially at large $\beta$.

\section{Outlook}

In this paper, we provided a unified analysis of a general class of algorithms for spectrum and spectral sum approximation. 
Our bounds are based on separately analyzing the approximation of the CESM $\Phi$ by a random weighted CESM $\Psi$, and then approximating $\Psi$ by a polynomial quadrature rule. 
Our bounds are mostly stated in terms of the best approximation on a given interval $[\lmin,\lmax]$. 
However, it is often the case that Lanczos-based methods significantly outperform such bounds. 
Thus, a study of spectrum dependent/a posteriori bounds which are suitable for use as stopping criteria would be of practical interest.
We also provided a variety of illustrative numerical experiments which highlight some of the peculiarities of matrix free quadrature.
We believe there is value in further study of how to choose parameters for the \cref{alg:protoalg} based on application specific knowledge.

\printbibliography


\end{document}


\maketitle

\begin{abstract}
We provide an overview of the relevant matrix-free quadrature algorithms and discuss potential practical implementations. 
In particular, we show that parameters used by the kernel polynomial method can be selected after the main computation has occurred.
The algorithms computational properties and behavior in finite precision arithmetic are also discussed and compared.
\end{abstract}

\section{Introduction}

The algorithms studied in the main paper were only described mathematically.
In this supplement, we discuss some potential implementations which provides further insight into the connections between the algorithms. 
In particular, we discuss the behavior of SLQ in finite precision arithmetic, and argue that a loss of orthogonality in the Lanczos algorithm does not cause SLQ to stop converging.
We also provide additional numerical experiments.

\subsection{Notation}

We use \emph{zero-indexed numpy style slicing} to indicate entries of matrices and vectors.
Thus, the submatrix consisting of rows \( r \) through \( r'-1 \) and columns \( c \) through \( c'-1 \) by \( [\vec{B}]_{r:r',c:c'} \). 
If any of these indices are equal to \( 0 \) or \( n \), they may be omitted and if \( r' = r+1  \) or \( c' = c+1 \), then we will simply write \( r \) or \( c \).
Thus, \( [\vec{B}]_{:,:2} \) denotes the first two columns of \( \vec{B} \), and \( [\vec{B}]_{3,:} \) denotes the fourth row (corresponding to index 3) of \( \vec{B} \).

\subsection{Orthogonal polynomials}

In this section, we provide only the needed background on orthogonal polynomials. 
For a more thorough treatment, see \cite{szego_39,gautschi_04}.

Throughout this paper \( \mu \) will be a non-negative probability distribution function. 
Associated with \( \mu \) is the inner product \( \langle \cdot, \cdot \rangle_\mu \) defined by
\begin{equation}
    \langle f, g\rangle_\mu := \int fg \d\mu.
\end{equation}
If we apply Gram--Schmidt sequentially to the monomial basis \( \{ x^i \}_{i=0}^{\infty} \), as long as the monomials are integrable, we obtain a sequence \( \{ p_i \}_{i=0}^{\infty} \) of polynomials with \( \deg(p_i) = i \) which are orthonormal with respect to the inner product \( \langle \cdot ,\cdot \rangle_\mu \); i.e. \( \langle p_i, p_j \rangle_\mu = \bOne(i=j) \) and can be chosen to have positive leading coefficient. 
The polynomials \( \{ p_i \}_{i=0}^{\infty} \) are called the (normalized) orthogonal polynomials with respect to \( \mu \), and we denote the roots of \( p_{i} \) by \( \{ \theta_{j}^{(i)} \}_{j=1}^{i} \).

\begin{definition}
    The modified moments of \( \Psi \) with respect to \( \mu \) are 
    \begin{equation}
        m_i := \int p_i \d\Psi
        ,\qquad
        i=0,1,\ldots
    \end{equation}
    If \( m_0, \ldots, m_{s} < \infty \), we say \( \Psi \) has finite moments through degree \( s \).
\end{definition}
Note that for any \( i \geq 0 \) the modified moment \( m_i \) for \( \Psi \) (with respect to \( \mu \)) can be written as the quadratic form
\begin{equation}
    m_i = \int p_i \d \Psi
    = \vec{v}^\cT p_i(\vec{A}) \vec{v}
\end{equation}
and can therefore be computed from the information in \( \mathcal{K}_s(\vec{A},\vec{v}) \) for any \( s \geq i/2 \).

It is well known that the orthogonal polynomials with respect to \( \mu \) satisfy a three term recurrence
\begin{equation}
    \label{sup:eqn:poly_three_term}
    x p_i(x) =  \beta_{i-1} p_{i-1}(x) + \alpha_i p_i(x) + \beta_{i} p_{i+1}(x)
    ,\qquad i =0,1,\ldots
\end{equation}
with initial conditions \( p_0(x) = 1 \), \( p_{-1}(x) = 0 \), and \( \beta_{-1} = 0 \) for some coefficients \( \{ \alpha_i \}_{i=0}^{\infty} \) and \( \{ \beta_i \}_{i=0}^{\infty} \) chosen to ensure orthonormality.
Typically, these coefficients are placed in a, possibly semi-infinite, tridiagonal matrix \( \vec{M} \) called the Jacobi matrix corresponding to \( \mu \); i.e.,
\begin{equation}
    \vec{M} := 
    \operatorname{tridiag}
    \left(\hspace{-.75em} \begin{array}{c}
        \begin{array}{cccc} \beta_0 & \beta_1 & \cdots & \end{array} \\
        \begin{array}{ccccc} \alpha_0 & \alpha_1 & \alpha_2 & \cdots & \end{array} \\
        \begin{array}{cccc} \beta_0 & \beta_1 & \cdots &  \end{array} 
    \end{array} \hspace{-.75em}\right).
\end{equation}
For a distribution \( \nu \), different from \( \mu \), we will denote the corresponding Jacobi matrix by \( \vec{M}(\nu) \). 

There is a one-to-one correspondence between the upper-leftmost \( k\times k \) submatrix of a Jacobi matrix and the moments through degree \( 2k-1 \) of the associated distribution function.
In fact, up to a constant, \( p_i \) is the characteristic polynomial of the upper-leftmost \( i\times i \) principal submatrix \( [\vec{M}]_{:i,:i} \) of \( \vec{M} \).
Recalling that we have denoted the zeros of \( p_{s+1} \) by \( \{ \theta_j^{(s+1)} \}_{j=1}^{s+1} \), we observe that the \( \{ \theta_j^{(s+1)} \}_{j=1}^{s+1} \) are distinct and, as a direct consequence of the three term recurrence \cref{sup:eqn:poly_three_term}, 
\begin{equation}
    \label{sup:eqn:jacobi_evec}
    \big[ p_0(\theta_{i}^{(s+1)}), p_1(\theta_{i}^{(s+1)}), \ldots, p_{s}(\theta_{i}^{(s+1)}) \big]^\cT
\end{equation}
is an eigenvector of \( [\vec{M}]_{:s+1,:s+1} \) with eigenvalue \( \theta_{i}^{(s+1)} \) \cite{golub_meurant_09}.

Owing to the deep connection between Chebyshev polynomials and approximation theory \cite{trefethen_19}, one particularly important choice of \( \mu \) is the distribution function corresponding to the Chebyshev polynomials of the first kind. 
We will often treat this case with special care.\footnote{While we treat the case of the Chebyshev polynomials of the first kind, similar results should be expected for other classical orthogonal polynomials.}
\begin{definition}
    The Chebyshev distribution function of the first kind, \( \mu_{a,b}^T : [a,b] \to [0,1] \), is
    \begin{equation}
        \mu_{a,b}^T(x) := \int_{a}^{x}\frac{2}{\pi(b-a)} \frac{1}{\sqrt{1-\big(\frac{2}{b-a}t - \frac{b+a}{b-a}\big)^2}} \d{t}
        = \frac{1}{2} + \frac{1}{\pi}\arcsin\left(\frac{2}{b-a}x - \frac{b+a}{b-a}\right).
    \end{equation}
\end{definition}

\begin{definition}
    The Chebyshev polynomials of the first kind, denoted \( \{ T_i \}_{i=0}^{\infty} \), are defined by the recurrence \( T_0(x) = 1 \), \( T_1(x) = x \), and
    \begin{equation}
        T_{i+1}(x) := 2 x T_i(x) - T_{i-1}(x)
        ,\qquad i=2,3,\ldots.
    \end{equation}
\end{definition}

It can be verified that the orthogonal polynomials \( \{p_i\}_{i=0}^{\infty} \) with respect to \( \mu_{a,b}^T \) are given by \( p_0(x) = T_0(x) = 1 \) and 
    \begin{equation}
        p_i(x) = \sqrt{2} T_i \left( \frac{2}{b-a} x - \frac{b+a}{b-a} \right)
        ,\qquad i=1,2,\ldots.
    \end{equation}
    Therefore, the Jacobi matrix \( \vec{M}( \mu_{a,b}^T ) \) is 
    \begin{equation}
    \vec{M}(\mu_{a,b}^T) =
        \operatorname{tridiag}
        \left(\hspace{-.75em} \begin{array}{c}
        \begin{array}{cccc} \frac{1}{2\sqrt{2}}(b-a) & \frac{1}{4}(b-a) & \cdots & \end{array} \\
        \begin{array}{ccccc} \frac{1}{2}(a+b) & \frac{1}{2}(a+b) & \frac{1}{2}(a+b) & \cdots & \end{array} \\
        \begin{array}{cccc} \frac{1}{2\sqrt{2}}(b-a) & \frac{1}{4}(b-a) & \cdots &  \end{array} 
    \end{array} \hspace{-.75em}\right).
    \end{equation}

\section{Extracting moments from a Krylov subspace}
\label{sup:sec:krylov_moments}

The main computational cost of the prototypical algorithm discussed in this paper is generating the Krylov subspaces \( \{ \mathcal{K}_k(\vec{A},\vec{v}_{\ell}) \}_{\ell=1}^{\nv} \) and extracting the modified moments of each \( \Psi_\ell \) with respect to \( \mu \).
We now describe several approaches for this task in this section.
Specifically, we describe the use of an explicit polynomial recurrence in \cref{sup:sec:moments_direct}, the Lanczos algorithm in \cref{sup:sec:lanczos}, and a general approach using connection coefficients in \cref{sup:sec:connection_coeffs}. 
In classical approximation theory, quadrature rules are most commonly used to approximate integrals against some simple weight function; for instance, a weight function for some family of classical orthogonal polynomials. 
As a result, it is often the case that the modified moments can be computed very efficiently \cite{townsend_webb_olver_17}.
In contrast, the weighted CESM depends on the eigenvalues of \( \vec{A} \) which are unknown and may be distributed non-uniformly.

\begin{remark}
Since we have defined \( \Psi \) and \( \vec{v} \) as random variables, the outputs of the algorithms from this section are also random variables. 
The algorithms we describe are deterministic when conditioned on the sample of \( \vec{v} \) (and therefore \( \Psi \)), and any statement which involves \( \Psi \) or \( \vec{v} \) not enclosed in \( \PP[ \: \cdot \:] \) should be interpreted to hold with probability one. 
In particular, such statements will hold for any sample of \( \vec{v} \) which has nonzero projection onto each eigenvector of \( \vec{A} \).
\end{remark}

\subsection{Computing modified moments directly}
\label{sup:sec:moments_direct}

Perhaps the most obvious approach to computing modified moments is to construct the basis \( [ p_0(\vec{A})\vec{v}, \ldots, p_{k}(\vec{A})\vec{v}] \) for \( \mathcal{K}_{k}(\vec{A},\vec{v}) \) and then compute 
\begin{equation}
    [m_0, m_1, \ldots, m_k] = 
    \vec{v}^\cT \big[ p_0(\vec{A})\vec{v}, \ldots, p_{k}(\vec{A})\vec{v}\big].
\end{equation}
This can be done using \( k \) matrix-vector products and \( O(n) \) storage using the matrix recurrence version of \cref{sup:eqn:poly_three_term}.
Indeed, we have that
\begin{equation}
    \vec{A} p_i(\vec{A}) \vec{v} =  \beta_{i-1} p_{i-1}(\vec{A})\vec{v} + \alpha_i p_i(\vec{A})\vec{v} + \beta_{i} p_{i+1}(\vec{A})\vec{v}
\end{equation}
from which we can implement an efficient matrix-free algorithm to compute the modified moments \( \{ m_i \}_{i=0}^{k} \), as shown in \cref{alg:moments}.

\begin{labelalgorithm}[H]{moments}{modified-moments}{Get modified moments wrt. \( \mu \) of weighed CESM}
\begin{algorithmic}[1]
    \Procedure{\thealgorithmname}{$\vec{A}, \vec{v}, k, \mu$}
    \State \( \vec{q}_0 = \vec{v} \), \( m_0 = \vec{v}^\cT\vec{v} \), \( \vec{q}_{-1} = \vec{0} \), \( \beta_{-1} = 0 \)
    \For {\( i=0,1,\ldots, k-1 \)}
        \State \( \vec{q}_{i+1} = \frac{1}{\beta_{i}} \left( \vec{A} \vec{q}_i - \alpha_i \vec{q}_i - \beta_{i-1} \vec{q}_{i-1} \right) \)
        \State \( m_{i+1}  = \vec{v}^\cT \vec{q}_{i+1} \)
    \EndFor
    \State \Return \( \{ m_i \}_{i=0}^{k} \)
\EndProcedure
\end{algorithmic}
\end{labelalgorithm}

If we instead compute 
\begin{equation}
    \big[ p_0(\vec{A})\vec{v}, \ldots, p_{k}(\vec{A})\vec{v}\big]^\cT
    \big[ p_0(\vec{A})\vec{v}, \ldots, p_{k}(\vec{A})\vec{v}\big],
\end{equation}
then we have the information required to compute the modified moments though degree \( 2k \).
However, it is not immediately clear how to do this without the \( O(kn) \) memory required to store a basis for \( \mathcal{K}_k(\vec{A},\vec{v}) \).
It turns out it is indeed generally possible to compute these moments without storing the entire basis, and we discuss a principled approach for doing so using connection coefficients in \cref{sup:sec:connection_coeffs}.

One case where a low-memory method for extracting the moments to degree \( 2k \) is straightforward is when \( \mu = \mu_{a,b}^T \).
This is because, for all \( i \geq 0 \), the Chebyshev polynomials satisfy the identities
\begin{equation}
    T_{2i}(x) = 2 (T_i(x))^2 - 1 
    ,\qquad
    T_{2i+1}(x) = 2 T_i(x) T_{i+1}(x) - x. 
\end{equation}
Thus, using the recurrence for the Chebyshev polynomials and their relation to the orthogonal polynomials with respect to \( \mu_{a,b}^T \), we obtain \cref{alg:cheb_moments}.
This algorithm is well-known in papers on the kernel polynomial method are variants; see for instance \cite{skilling_89,silver_roder_94,weisse_wellein_alvermann_fehske_06,hallman_21}.
 
\begin{remark}   
The Chebyshev polynomials grow rapidly outside the interval \( [-1,1] \).
Therefore, if the spectrum of \( \vec{A} \) extends beyond this interval, then computing the Chebyshev polynomials in \( \vec{A} \) may suffer from numerical instabilities.
Instead, the distribution function \( \mu_{a,b}^T \) and corresponding orthogonal polynomials should be used for some choice of \( a \) and \( b \) with \( [\lmax,\lmax] \subset [a,b] \).

\end{remark}

\begin{labelalgorithm}[H]{cheb_moments}{Chebyshev-moments}{Get modified (Chebyshev) moments of weighed CESM}
\begin{algorithmic}[1]
    \Procedure{\thealgorithmname}{$\vec{A}, \vec{v}, k, a,b$}
    \State \( \vec{q}_0 = \vec{v} \), \( m_0 = (\vec{q}_0)^\cT \vec{q}_0 \)
    \State \( \vec{q}_1 = \frac{2}{b-a}(\vec{A} \vec{q}_0 - \frac{a+b}{2} \vec{q}_0) \), \( m_1 =  \sqrt{2} (\vec{q}_0)^\cT \vec{q}_1 \)
    \For {\( i=1,2,\ldots, k-1 \)}
        \State \( m_{2i}  = \sqrt{2} (2(\vec{q}_i)^\cT \vec{q}_{i} - m_0) \)
        \State \( \vec{q}_{i+1} = 2 \frac{2}{b-a} (\vec{A} \vec{q}_i - \frac{a+b}{2} \vec{q}_i) - \vec{q}_{i-1} \)
        \State \( m_{2i+1}  = \sqrt{2} (2(\vec{q}_i)^\cT \vec{q}_{i+1}) - m_1 \)
    \EndFor
    \State \( m_{2k} = \sqrt{2}( 2 (\vec{q}_k)^\cT\vec{q}_k - m_0) \)
    \State \Return \( \{ m_i \}_{i=0}^{2k} \)
\EndProcedure
\end{algorithmic}
\end{labelalgorithm}

\subsection{Lanczos algorithm}
\label{sup:sec:lanczos}

Thus far, we have set \( \mu \) to be some known distribution such that the Jacobi matrix \( \vec{M} \) and the corresponding orthogonal polynomials are known.
In some situations, it will be useful to set \( \mu = \Psi \). 
However, in this case, the orthogonal polynomial recurrence itself is unknown and must be computed. 
This can be done by the Stieltjes procedure, which is mathematically equivalent to the well known Lanczos algorithm run on \( \vec{A} \) and \( \vec{v} \).
The Lanczos algorithm is shown in \cref{alg:lanczos}.
When run for \( k \) matrix vector products, Lanczos computes \( [ \alpha_0, \alpha_1, \ldots, \alpha_{k-1} ] \) and \( [\beta_0, \beta_1, \ldots, \beta_{k-1}] \), the coefficients for the diagonal and off diagonal of the Jacobi matrix $\vec{M}(\Psi)$.
We will denote this special Jacobi matrix by $\vec{T}$.

\begin{labelalgorithm}[H]{lanczos}{Lanczos}{Lanczos algorithm}
\begin{algorithmic}[1]
    \Procedure{\thealgorithmname}{$\vec{A}, \vec{v}, k$}
    \State \( \vec{q}_0 = \vec{v} / \|\vec{v}\| \), \( \beta_{-1} = 0 \), \( \vec{q}_{-1} = \vec{0} \)
    \For {\( i=0,1,\ldots, k-1 \)}
        \State \( \tilde{\vec{q}}_{i+1} = \vec{A} \vec{q}_{i} - \beta_{i-1} \vec{q}_{i-1} \)
        \State \( \alpha_{i} = \langle \tilde{\vec{q}}_{i+1}, \vec{q}_i \rangle \)
        \State \( \hat{\vec{q}}_{i+1} = \tilde{\vec{q}}_{i+1} - \alpha_{i} \vec{q}_i \)
        \State optionally, reorthogonalize,  \( \hat{\vec{q}}_{i+1} \) against \( \{\vec{q}_j\}_{j=0}^{i-1} \)
        \State \( \beta_{i} = \| \hat{\vec{q}}_{i+1} \| \)
        \State \( \vec{q}_{i+1} = \hat{\vec{q}}_{i+1} / \beta_{i} \)
    \EndFor
    \State \Return \( [\vec{T}]_{:k,:k} \) or \( [\vec{T}]_{:k+1,:k} \) 
\EndProcedure
\end{algorithmic}
\end{labelalgorithm}

\subsection{Connection coefficients to compute more modified moments}
\label{sup:sec:connection_coeffs}

We now discuss how to use connection coefficients to compute the modified moments of \( \Psi \) with respect to \( \mu \) given knowledge of either
(i) the modified moments of \( \Psi \) with respect to some distribution \( \nu \) or 
(ii) the tridiagonal matrix computed using the Lanczos algorithm (\cref{alg:lanczos}).
Much of our discussion on connection coefficients is based on \cite{webb_olver_21}; see also \cite{fischer_golub_91}.

\begin{definition}
    \label[definition]{def:connection_coeffs}
    The connection coefficient matrix \( \vec{C} = \vec{C}_{\mu\to\nu} \) is the upper triangular matrix representing a change of basis between the orthogonal polynomials \( \{ p_i \}_{i=1}^{\infty} \) with respect to \( \mu \) and the orthogonal polynomials \( \{ q_i \}_{i=1}^{\infty} \) with respect to \( \nu \), whose entries satisfy,
    \begin{equation}
        p_s = [\vec{C}]_{0,s} q_0 +  [\vec{C}]_{1,s} q_1 +  \cdots + [\vec{C}]_{s,s} q_s.
    \end{equation}
\end{definition}

\Cref{def:connection_coeffs} implies that
\begin{equation}
    m_i =
    \int p_i \d\Psi
    = 
    \sum_{j=0}^{i} [\vec{C}]_{j,i} \int q_j \d\Psi
    =
    \sum_{j=0}^{i} [\vec{C}]_{j,i} n_j
    ,\qquad i=0,1,\ldots, s
\end{equation}
where \( \{n_i\}_{i=0}^{s} \) are the modified moments of \( \Psi \) with respect to \( \nu \).
Thus, we can easily obtain the modified moments \( \{ m_i \}_{i=0}^{s} \) of \( \Psi \) with respect to \( \mu \) from the modified moments of \( \Psi \) with respect to \( \nu \).
In particular, if \( \vec{m} \) and \( \vec{n} \) denote the vectors of modified moments of $\mu$ and $\nu$ respectively, then \( \vec{m} = \vec{C}^\T \vec{n} \).

Moreover, in the special case that \( \nu \) has the same moments as \( \Psi \) through degree \( s \), 
\begin{equation}
    n_j = \int q_j \d\Psi 
    = \int q_j \d\nu
    = \int q_0 q_j \d\nu 
    = \bOne(j = 0).
\end{equation}
Therefore, the modified moments of \( \Psi \) (with respect to \( \mu \)) through degree \( s \) can be computed by
\begin{equation}
    m_i=
    \int p_i \d\Psi = 
    [ \vec{C}]_{0,i}
    ,\qquad i =0,1,\ldots, s.
\end{equation}

In order to use the above expressions, we must compute the connection coefficient matrix.
\Cref{def:connection_coeffs} implies that for all \( i\leq j \), the entries of the connection coefficient matrix are given by
\begin{equation}
    [\vec{C}]_{i,j} = \int q_i p_j \d\nu.
\end{equation}    
Unsurprisingly, then, the entries of the connection coefficient matrix \( \vec{C} = \vec{C}_{\mu\to\nu} \) can be obtained by a recurrence relation.
\begin{proposition}[{\cite[Corollary 3.3]{webb_olver_21}}]
\label[proposition]{thm:connection_coefficient_recurrence}
    Suppose the Jacobi matrices for \( \mu \) and \( \nu \) are respectively given by
    \begin{equation}
        \vec{M}(\mu) =
            \operatorname{tridiag}
    \left(\hspace{-.75em} \begin{array}{c}
        \begin{array}{cccc} \beta_0 & \beta_1 & \cdots & \end{array} \\
        \begin{array}{ccccc} \alpha_0 & \alpha_1 & \alpha_2 & \cdots & \end{array} \\
        \begin{array}{cccc} \beta_0 & \beta_1 & \cdots &  \end{array} 
    \end{array} \hspace{-.75em}\right)
    ,\qquad
        \vec{M}(\nu) =
        \operatorname{tridiag}
    \left(\hspace{-.75em} \begin{array}{c}
        \begin{array}{cccc} \delta_0 & \delta_1 & \cdots & \end{array} \\
        \begin{array}{ccccc} \gamma_0 & \gamma_1 & \gamma_2 & \cdots & \end{array} \\
        \begin{array}{cccc} \delta_0 & \delta_1 & \cdots &  \end{array} 
    \end{array} \hspace{-.75em}\right).
    \end{equation}

    Then the entries of \( \vec{C} = \vec{C}_{\mu\to\nu} \) satisfy, for \( i,j \geq 0 \), the following recurrence:
\begin{align*}
    [\vec{C}]_{0,0} &= 1 
    \\
    [\vec{C}]_{0,1} &= (\gamma_0 - \alpha_0)/\beta_0
    \\
    [\vec{C}]_{1,1} &= \delta_0 / \beta_0 
    \\
    [\vec{C}]_{0,j} &= ((\gamma_0 - \alpha_{j-1}) [\vec{C}]_{0,j-1} + \delta_0 [\vec{C}]_{2,j-1} - \beta_{j-2} [\vec{C}]_{0,j-2} ) / \beta_{j-1}
    \\
    [\vec{C}]_{i,j} &= (\delta_{i-1} [\vec{C}]_{i-1,j-1} + (\gamma_i - \alpha_{j-1}) [\vec{C}]_{i,j-1} + \delta_i[\vec{C}]_{i+1,j-1} - \beta_{j-2} [ \vec{C}]_{i,j-2}) / \beta_{j-1}.
\end{align*}
\end{proposition}

\Cref{thm:connection_coefficient_recurrence} yields a natural algorithm for computing the connection coefficient matrix \( \vec{C}_{\mu\to\nu} \).
This algorithm is shown as \cref{alg:connection_coeffs}.
Note that \( \vec{C} \) is, by definition, upper triangular, so \( [\vec{C}]_{i,j}  = 0 \) whenever \( i > j \).
We remark that for certain cases, particularly transforms between the Jacobi matrices of classical orthogonal polynomials, faster algorithms are known \cite{townsend_webb_olver_17}.
We do not focus on such details in this paper, as the cost of products with \( \vec{A} \) is typically far larger than the cost of computing \( \vec{C}_{\mu\to\nu} \).

\begin{labelalgorithm}[H]{connection_coeffs}{connection-coeffs}{Get connection coefficients}
\begin{algorithmic}[1]
    \Procedure{\thealgorithmname}{$\mu,k_{m_1},k_{m_2},\nu,k_{n_1},k_{n_2}$}
    \State \( [\vec{C}]_{0,0} = 1 \) 
    \For {$j = 1,\ldots, \min(k_{m_2},k_{n_1}+k_{n_2})$}
    \For {$i=0,\ldots, \min(j,k_{n_1}+k_{n_2}-j)$}
    \State 
    \(
    \begin{aligned}
        [\vec{C}]_{i,j} = (&\delta_{i-1} [\vec{C}]_{i-1,j-1} + (\gamma_i - \alpha_{j-1}) [\vec{C}]_{i,j-1} 
        \\&+ \delta_i[\vec{C}]_{i+1,j-1} 
        - \beta_{j-2} [ \vec{C}]_{i,j-2}) / \beta_{j-1} 
    \end{aligned}
    \)
    \Comment{ \( \begin{aligned}\text{if \( i' > j' \) or \( j' = -1 \), drop} \\ \text{the term involving \( [\vec{C}]_{i',j'} \)} \end{aligned}\)}
    \EndFor
    \EndFor
    \State \Return \( \vec{C} = \vec{C}_{\mu\to\nu} \)
\EndProcedure
\end{algorithmic}
\end{labelalgorithm}

\begin{remark}
    From \cref{thm:connection_coefficient_recurrence} it is not hard to see that \( [\vec{C}]_{:k,:k} \) can be computed using \( [\vec{M}(\nu)]_{:k,:k} \) and \( [\vec{M}(\mu)]_{:k,:k} \). 
    Moreover, \( [\vec{C}]_{0,:2k+1} \) can be computed using \( [\vec{M}(\nu)]_{:k+1,:k} \) and \( [\vec{M}(\mu)]_{:2k,:2k} \). 
    In general \( \vec{M}(\mu) \) will be known fully, and in such cases, the modified moments through degree \( 2k \) can be computed using the information generated by Lanczos run for \( k \) iterations.
\end{remark}

We can use \cref{alg:connection_coeffs} in conjunction with \cref{alg:cheb_moments} and \cref{alg:lanczos} to compute modified moments with respect to \( \mu \).
This is shown in \cref{alg:moments_cheb} and \cref{alg:moments_lanc} respectively.
\begin{labelalgorithm}[H]{moments_cheb}{moments-from-Cheb}{Get modified moments wrt. \( \mu \) of weighed CESM (via Chebyshev moments)}
\begin{algorithmic}[1]
    \Procedure{\thealgorithmname}{$\vec{A}, \vec{v}, s, \mu, a, b$}
    \State \( k = \lceil s/2 \rceil \)
    \State \( \{ n_i \}_{i=0}^{2k} = \ref{alg:name:cheb_moments}(\vec{A},\vec{v}, k,a,b)  \)
    \State \( \vec{C} = \ref{alg:name:connection_coeffs}([\vec{M}(\mu)]_{:2k,:2k},[\vec{M}(\mu_{a,b}^T)]_{:2k,:2k}) \)
    \For{\( i=0,1,\ldots, s \)}
        \State \( m_i = \sum_{j=0}^{i} [\vec{C}]_{j,i} n_j \)
    \EndFor
    \State \Return \( \{ m_i \}_{i=0}^{s} \)
\EndProcedure
\end{algorithmic}
\end{labelalgorithm}

\begin{labelalgorithm}[H]{moments_lanc}{moments-from-Lanczos}{Get modified moments wrt. \( \mu \) of weighed CESM (via Lanczos)}
\begin{algorithmic}[1]
    \Procedure{\thealgorithmname}{$\vec{A}, \vec{v}, s, \mu$}
    \State \( k = \lceil s/2 \rceil \)
    \State \( [\vec{T}]_{i:k+1,:k} = \ref{alg:name:lanczos}(\vec{A},\vec{v}, k)  \)
    \State \( \vec{C} = \ref{alg:name:connection_coeffs}([\vec{M}(\mu)]_{:2k,:2k},[\vec{T}]_{:k+1,:k}) \)
    \For{\( i=0,1,\ldots, s \)}
        \State \( m_i = [\vec{C}]_{0,i} \)
    \EndFor
    \State \Return \( \{ m_i \}_{i=0}^{s} \)
\EndProcedure
\end{algorithmic}
\end{labelalgorithm}

We conclude this section by noting that, given the modified moments of \( \Psi \) with respect to \( \mu \), the tridiagonal matrix produced by \cref{alg:lanczos} can itself be obtained \cite{gautschi_70,sack_donovan_71,wheeler_blumstein_72,gautschi_86}.
Special care must be taken in finite precision arithmetic as the maps from moments to quadrature nodes and weights or to the tridiagonal Jacobi matrix may be severely ill-conditioned \cite{gautschi_04,oleary_strakos_tichy_07}.

\section{Quadrature approximations for weighted spectral measures}
\label{sup:sec:quadrature}

We can use the information extracted by the algorithms in the previous section to obtain quadrature rules for the weighted CESM \( \Psi \).
We begin with a discussion on quadrature by interpolation in \cref{sup:sec:iq} followed by a discussion on Guassian quadrature (SLQ) in \cref{sup:sec:gq} and quadrature by approximation (KPM) in \cref{sup:sec:aq}.
Finally, in \cref{sup:sec:damped}, we describe how damping can be used to ensure the positivity of quadrature approximations.

\subsection{Quadrature by interpolation}
\label{sup:sec:iq}

Our first class of quadrature approximations for \( \Psi \) is the degree \( s \) quadrature by interpolation \( \qq[i]{\Psi}{s} \) (i.e. $\circ = \textrm{i}$) which is defined by the relation
\begin{equation}
    \label{sup:eqn:interp_f}
    \int f \d\qq[i]{\Psi}{s}
    := \int \ff[i]{f}{s} \d\Psi,
\end{equation}
where \( \ff[i]{f}{s} \) is the degree \( s \) polynomial interpolating a function \( f \) at the zeros \( \{\theta_{j}^{(s+1)}\}_{j=1}^{s+1} \) of \( p_{s+1} \), the degree $s+1$ orthogonal polynomial with respect to \( \mu \).
It's not hard to see that \cref{sup:eqn:interp_f} implies
\begin{equation}
    \qq[i]{\Psi}{s}= \sum_{j=1}^{s+1} \omega_{j}^{} \bOne\big[\theta_{j}^{(s+1)} \leq x \big],
\end{equation}
where the weights \( \{ \omega_{j}^{} \}_{j=1}^{s+1} \) are chosen such that the moments of \( \qq[i]{\Psi}{s} \) agree with those of \( \Psi \) through degree \( s \).

One approach to doing this is by solving the Vandermonde-like linear system of equations
\begin{equation}
    \label{sup:eqn:vandermonde}
    \vec{P}\bm{\omega} = \vec{m} 
    ,\qquad
    [\vec{P}]_{i,j} = p_i(\theta_{j}^{(s+1)})
    ,\qquad [\vec{m}]_i = \int p_i \d\Psi.
\end{equation}
This will ensure that polynomials of degree up to \( s \) are integrated exactly. 

While it is not necessary to restrict the test polynomials to be the orthogonal polynomials \( \{ p_i \}_{i=1}^{\infty} \) with respect to \( \mu \) nor the interpolation nodes to be the zeros \( \{ \theta_j^{(s+1)} \}_{j=1}^{s+1} \) of \( p_{s+1} \), doing so has several advantages.
If arbitrary polynomials are used, the matrix \( \vec{P} \) may be exponentially ill-conditioned; i.e. the condition number of the matrix could grow exponentially in \( k \).
This can cause numerical issues with solving \( \vec{P} \bm{\omega} = \vec{m} \).
If orthogonal polynomials are used, then as in \cref{sup:eqn:jacobi_evec} we see that the columns of \( \vec{P} \) are eigenvectors of the Jacobi matrix \( \vec{M} \).
Since \( \vec{M} \) is symmetric, this implies that \( \vec{P} \) has orthogonal columns; i.e. \( \vec{P}^\cT\vec{P} \) is diagonal.
Therefore, we can easily apply \( \vec{P}^{-1} \) through a product with \( \vec{P}^\cT \) and an appropriately chosen diagonal matrix.
In particular, if \( \vec{S} \) is the orthonormal matrix of eigenvectors, then \( \vec{P} = \vec{S} \operatorname{diag}([\vec{S}]_{0,:})^{-1} \).
From this, we see that \( \vec{P}^{-1} = \operatorname{diag}([\vec{S}]_{0,:})\vec{S}^\cT \).
This yields \cref{alg:iq} which, to the best of our knowledge, is not widely known.

\begin{labelalgorithm}[H]{iq}{quad-by-interp}{Quadrature by interpolation}
\begin{algorithmic}[1]
    \Procedure{\thealgorithmname}{$\{m_i\}_{i=0}^{s},\mu$}
    \State \( \bm{\theta},\vec{S} = \Call{eig}{[\vec{M}(\mu)]_{:s+1,:s+1}} \) \Comment{Eigenvectors normalized to unit length}
    \State \( \bm{\omega} = \operatorname{diag}([\vec{S}]_{0,:})\vec{S}^\cT \vec{m} \) 
    \State \Return \( \qq[i]{\Psi}{s} = \left( x\mapsto \sum_{j=1}^{s+1} \omega_{j}^{} \bOne(\theta_{j}\leq x) \right) \)
\EndProcedure
\end{algorithmic}
\end{labelalgorithm}

\begin{remark}   
In certain cases, such as \( \mu = \mu_{a,b}^T \), \( \vec{P}^{-1} \) can be applied quickly and stably using fast transforms, such as the discrete cosine transform, without ever constructing \( \vec{P} \).
\end{remark}

\subsection{Gaussian quadrature and stochastic Lanczos quadrature}
\label{sup:sec:gq}

While interpolation-based quadrature rules supported on \( k \) nodes do not, in general, integrate polynomials of degree higher than \( k-1 \) exactly, if we allow the nodes to be chosen adaptively we can do better.
The degree \( 2k-1 \) Gaussian quadrature rule \( \qq[g]{\Psi}{2k-1} \) for \( \Psi \) is obtained by constructing a quadrature rule by interpolation rule at the roots \( \{ \theta_{i}^{(k)} \}_{i=1}^{k} \) of the degree \( k \) orthogonal polynomial \( p_{k} \) of \( \Psi \) (i.e. by taking \( \mu = \Psi \)).
The Gaussian quadrature rule integrates polynomials of degree \( 2k-1 \) exactly; see \cite{golub_meurant_09}.

Because the polynomials \( \{ p_i \}_{i=0}^{\infty} \) are orthogonal with respect to the probability distribution \( \Psi \) function,
we have that
\begin{equation}
    m_i = \vec{v}^\cT p_i(\vec{A})\vec{v} 
    =
    \int p_i p_0 \d\Psi
    = \bOne(i=0).
\end{equation}
This means the right hand side \( \vec{m} \) of \cref{sup:eqn:vandermonde} is the first canonical unit vector \( \vec{e}_0 = [1,0,\ldots, 0]^\cT \).
Thus, as in \cref{alg:iq}, \( \vec{\omega} = \operatorname{diag}([\vec{S}]_{0,:}) [\vec{S}]_{0,:}\); that is, the quadrature weights are the squares of the first components of the unit length eigenvectors of \( [\vec{T}]_{:k,:k} \).
We then arrive at \cref{alg:gq} for obtaining a Gaussian quadrature rule for \( \Psi \) from the tridiagonal matrix \( [\vec{T}]_{:k,:k} \) generated by \cref{alg:lanczos}.
When used with the prototypical algorithm \cref{alg:protoalg}, the resulting algorithm is often called SLQ.

\begin{labelalgorithm}[H]{gq}{Gaussian-quadrature}{Gaussian quadrature}
\begin{algorithmic}[1]
    \Procedure{\thealgorithmname}{$[\vec{T}]_{:k,:k}$}
    \State \( \bm{\theta}, \vec{S} = \textsc{eig}([\vec{T}]_{:k,:k}) \) \Comment{Eigenvectors normalized to unit length}
    \State \( \bm{\omega} = \operatorname{diag}([\vec{S}]_{0,:}) [\vec{S}]_{0,:} \)
    \State \Return \( \qq[g]{\Psi}{2k-1} = \left( x\mapsto \sum_{j=1}^{k} \omega_j \bOne(\theta_j\leq x) \right) \)
\EndProcedure
\end{algorithmic}
\end{labelalgorithm}

\begin{remark}
    To construct a Gaussian quadrature, the three term recurrence for the orthogonal polynomials of \( \Psi \) must be determined. 
    Thus, the main computational cost is computing the tridiagonal matrix defining this recurrence. 
    However, due to orthogonality, all but the degree zero modified moments are zero so we do not need to compute the moments.
    This is in contrast to other schemes where the polynomial recurrence is known but the modified moments must be computed.
\end{remark}

\subsection{Quadrature by approximation}
\label{sup:sec:aq}

Rather than defining a quadrature approximation using an interpolating polynomial, we might instead define an approximation \( \qq[a]{\Psi}{s} \) by the relation
\begin{equation}
    \int f\d\qq[a]{\Psi}{s}
    := \int \ff[a]{f}{s} \d\Psi,
\end{equation}
where \( \ff[a]{f}{s} \) is the projection of \( f \) onto the orthogonal polynomials with respect to \( \mu \) through degree \( s \) in the inner product \( \langle \cdot, \cdot\rangle_\mu \).
That is,
\begin{equation}
    \ff[a]{f}{s}
    := \sum_{i=0}^{s} \langle f, p_i \rangle_\mu \:p_i
    = \sum_{i=0}^{s} \left(\int f p_i\d\mu \right) \:p_i.
\end{equation}
We can obtain an approximation to the density \( \d\Psi/\d{x} \) by using the density \( \d\mu/\d{x} \) and the definition of the Radon--Nikodym derivative: 
\begin{equation}
    \frac{\d\qq[a]{\Psi}{s}}{\d{x}} =  \frac{\d\qq[a]{\Psi}{s}}{\d\mu}  \frac{\d\mu}{\d{x}}
    = \frac{\d\mu}{\d{x}} \sum_{i=0}^{s} m_i p_i.
\end{equation}
Integrating this ``density'' gives the approximation \( \qq[a]{\Psi}{s} \) shown in \cref{alg:aq}. 
A spectrum adaptive Lanczos-based implementation is described in \cite{chen_23}. 

\begin{labelalgorithm}[H]{aq}{quad-by-approx}{Quadrature by approximation}
\begin{algorithmic}[1]
    \Procedure{\thealgorithmname}{$\{ m_i \}_{i=0}^{s}, \mu$}
    \State \Return \( \qq[a]{\Psi}{s} = \left( x\mapsto \sum_{i=0}^{s} m_i \int_{(-\infty,x]} p_i \d\mu \right) \)
\EndProcedure
\end{algorithmic}
\end{labelalgorithm}

\begin{remark}
    When \( \mu = \mu_{a,b}^T \), \( {\d\qq[a]{\Psi}{s}} / {\d\mu} \) can be evaluated quickly at Chebyshev nodes by means of the discrete cosine transform.
    This allows the density \( {\d\qq[a]{\Psi}{s}} / {\d{x}} \) to be evaluated quickly at these points.
\end{remark}

\subsubsection{Evaluating spectral sums and the relation to quadrature by interpolation}

We have written the output of \cref{alg:aq} as a distribution function for consistency.
However, note that 
\begin{equation}
    \int f \d \qq[a]{\Psi}{s}
    = \sum_{i=0}^{s} m_i \int f  p_i \d\mu.
\end{equation}
Thus, if used for the task of spectral sum approximation, the distribution function \( \qq[a]{\Psi}{s} \) need not be computed explicitly. 
Rather, the \( \mu \)-projections of \( f \) onto the orthogonal polynomials with respect to \( \mu \) can be used instead.
In many cases, the values of these projections are known analytically, and even if they are unknown, computing them is a scalar problem independent of the matrix size \( n \).

A natural approach to computing the \( \mu \)-projections of \( f \) numerically is to use a quadrature approximation for \( \mu \).
Specifically, we might use the \( d \)-point Gaussian quadrature rule \( \qq[g]{\mu}{2d-1} \) for \( \mu \) to approximate \( \int f p_i \d\mu \).
This gives us the approximation 
\begin{equation}
    \sum_{i=0}^{s} m_i \int f p_i \d\mu
    \approx \sum_{i=0}^{s} m_i \int f p_i \d \qq[g]{\mu}{2s+1}
    = \sum_{i=0}^{s} m_i \sum_{j=1}^{d} \omega_{i}^{(d)} f(\theta_{j}^{(d)}) p_i(\theta_j^{(d)})
\end{equation}
where \( \omega_i^{(d)} \) are the Gaussian quadrature weights for \( \mu \).

Similar to above, denote by \( \vec{P} \) the \( d\times d \) Vandermonde-like matrix of orthogonal polynomials with respect to \( \mu \) evaluated at the zeros of \( p_{d+1} \).
If \( \vec{S} \) is the orthonormal matrix of eigenvectors of the \( d\times d \) Jacobi matrix \( [\vec{M}]_{:d,:d} \) for \( \mu \), then recall that the Gaussian quadrature weights \( \bm{\omega}^{(d)} \) are given by \( \operatorname{diag}([\vec{S}]_{0,:s+1})([\vec{S}]_{:s+1,:})^\cT \).
This yields \cref{alg:aaq}.
Note that in the case \( d=s+1 \), \cref{alg:aaq} is equivalent to \cref{alg:iq}.
In other words, quadrature by interpolation can be viewed as a discretized version of quadrature by approximation.

\begin{labelalgorithm}[H]{aaq}{approximate-quad-by-approx}{Approximate quadrature by approximation}
\begin{algorithmic}[1]
    \Procedure{\thealgorithmname}{$\{m_i\}_{i=0}^{s},d,\mu$}
    \State \( \bm{\theta},\vec{S} = \Call{eig}{[\vec{M}(\mu)]_{:d,:d}} \) \Comment{Eigenvectors normalized to unit length}
    \State \( \bm{\omega} = \operatorname{diag}([\vec{S}]_{0,:s+1})([\vec{S}]_{:s+1,:})^\cT \vec{m} \) 
    \State \Return \( \qq[i]{\Psi}{s} = \left( x\mapsto \sum_{j=1}^{k} \omega_{j}^{} \bOne(x\leq \theta_{j}) \right) \)
\EndProcedure
\end{algorithmic}
\end{labelalgorithm}

\subsection{Positivity by damping and the kernel polynomial method}
\label{sup:sec:damped}

While it is clear that that the Gaussian quadrature rule \( \qq[g]{\Psi}{s} \) is a non-negative probability distribution function, neither \( \qq[i]{\Psi}{s} \) nor \( \qq[a]{\Psi}{s} \) are guaranteed to be weakly increasing.
We now discuss how to use damping to enforce this property.

Towards this end, define the damping kernel,
\begin{equation}
    P_x(y) 
    = \sum_{i=0}^{s} \rho_{i} p_i(x) p_i(y)
\end{equation}
where \( \{ \rho_i \}_{i=0}^{s} \) are \emph{damping coefficients}.
Next, define the damped interpolant \( \ff[d-i]{f}{s} \) and approximant \( \ff[d-a]{f}{s} \) by
\begin{equation}
    \ff[d-i]{f}{s}(x) = \int P_x f \d\mu
    \qquad \text{and} \qquad
    \ff[d-a]{f}{s}(x) = \int P_x f \d\mu.
\end{equation}
Note that if  \( \rho_{i} = 1 \) for all \( i \),  then \( \ff[d-i]{f}{s} = \ff[i]{f}{s} \) and \( \ff[d-a]{f}{s} = \ff[a]{f}{s} \).

These approximations induce  \( \qq[d-i]{\Psi}{s} \) and \( \qq[d-a]{\Psi}{s} \) by
\begin{equation}
    \int f \d\qq[d-i]{\Psi}{s}
    := \int \ff[d-i]{f}{s} \d\Psi
    \qquad\text{and}\qquad
    \int f \:  \d\qq[d-a]{\Psi}{s}
    := \int \ff[d-a]{f}{s} \d\Psi.
\end{equation}
Algorithmically, this is equivalent to replacing \( m_i \) by \( \rho_{i}^{} m_i \) in the expressions for interpolatory and quadrature by approximation described in \cref{sup:sec:iq,sup:sec:aq}.

Observe that
\begin{equation}
    \int P_x \d\mu
    = \int P_x \d\qq[g]{\mu}{2s+1}
    = \rho_{0}.
\end{equation}
Thus, \( \qq[d-i]{\Psi}{s} \) and \( \qq[d-a]{\Psi}{s} \) will have unit mass provided \( \rho_0 = 1 \).
Next, suppose \( P_x(y) \geq 0 \) for all \( x,y \) and that \( f \geq 0 \).
Then clearly \( \int f \d\qq[d-i]{\Psi}{s} \geq 0\) and \( \int f \d\qq[d-a]{\Psi}{s} \) so the approximations are also weakly increasing. 

Perhaps the most widely used damping kernel is induced by the Jackson coefficients
\begin{equation}
\label{sup:eqn:jackson_coeffs}
    \rho_{i}^{J}
    = \frac{(s-i+2)\cos \left( \frac{i \pi}{s+2} \right) + \sin \left( \frac{i \pi}{s+2} \right)\cot \left( \frac{\pi}{s+2} \right)}{s+2}
    ,\qquad
    i=0,1,\ldots, s.
\end{equation}
When used with \( \mu = \mu_{a,b}^T \) this kernel results in the standard version of the kernel polynomial method \cite{weisse_wellein_alvermann_fehske_06,braverman_krishnan_musco_22}.
Other damping schemes are discussed in \cite{weisse_wellein_alvermann_fehske_06,lin_saad_yang_16}.

\section{Qualitative comparison of algorithms}
\label{sup:sec:tradeoffs}supplement

\subsection{Computational costs}

In \cref{sup:sec:krylov_moments} we described \cref{alg:moments} (\ref{alg:name:moments}), \cref{alg:cheb_moments} (\ref{alg:name:cheb_moments}), \cref{alg:lanczos} (\ref{alg:name:lanczos}), \cref{alg:moments_cheb} (\ref{alg:name:moments_cheb}), and \cref{alg:moments_lanc} (\ref{alg:name:moments_lanc}) which are used to compute the information required for the quadrature approximations described in \cref{sup:sec:quadrature}. 
Since \cref{alg:moments_cheb,alg:moments_lanc} respectively call \cref{alg:cheb_moments,alg:lanczos}, \cref{alg:moments,alg:cheb_moments,alg:lanczos} constitute the bulk of the computational cost of all implementations of the protoalgorithm discussed in this paper.

In each iteration, \cref{alg:moments,alg:cheb_moments,alg:lanczos} each require one matrix vector product with \( \vec{A} \) along with several scalar multiplications, vector additions, and inner products.
As such, the total computational cost of each algorithms is \( O(k(T_{\textup{mv}} + n)) \) where \( k \) is the number of iterations and $T_{\textup{mv}}$ is the cost of matrix-vector product with $\vec{A}$.
Here we ignore terms depending only on \( k \) (e.g. \( k^2 \)) which are low-order if we assume \( k\ll n \).
Each of the algorithms can also be implemented using just \( O(n) \) storage; i.e. without storing the entire basis for the Krylov subspace which would cost \( O(kn) \) storage.

While the algorithms are typically quite storage efficient, there are some situations in which it may be desirable to store the whole Krylov basis.
First, \cref{alg:lanczos} is sometimes run with full reorthogonalization. 
This can improve numerical stability, but increases the computation cost to \( O(k(T_{\textup{mv}} + kn)) \).
Next, by delaying all inner products to a final iteration (or using a non-blocking implementation), the number of global reductions required by \cref{alg:moments} and \cref{alg:cheb_moments} can be reduced.
Since global communication can significantly slow down Krylov subspace methods, this may speed up computation on highly parallel machines \cite{dongarra_heroux_luszczek_15}.
As mentioned earlier, there are implementations of the Lanczos algorithm which aim to decrease the number of global communications \cite{carson_demmel_15,carson_20}.
Designing Krylov subspace methods for avoiding or reducing communication costs is a large area study, but further discussion is outside the scope of this paper. 

Finally, because all of the \( n_{\text{v}} \) quadrature approximations are carried out independently, the protoalgorithm is highly parallelizable.
The algorithms can also be implemented to run on groups of starting vectors simultaneously using blocked matrix-vector products, at the cost of increasing the required memory by a factor of the block-size.
In many practical situations, particularly those where \( \vec{A} \) is extremely large, the dominant cost of a matrix-vector product is due to the data access pattern to \( \vec{A} \) which is not changed in a blocked matrix-vector product.
Thus, a block matrix-vector product can often be performed in essentially the same time as a single matrix-vector product.

\subsection{Stability in finite precision arithmetic}

When implemented exactly, each of \cref{alg:moments,alg:cheb_moments,alg:lanczos} compute \( \vec{q}_{i+1} \) by a symmetric three term recurrence of the form
\begin{equation}
    \vec{q}_{i+1} = \frac{1}{\beta_{i}} \left( \vec{A} \vec{q}_i - \alpha_i \vec{q}_i - \beta_{i-1} \vec{q}_{i-1}  \right).
\end{equation}
In \cref{alg:moments,alg:cheb_moments} the coefficients are predetermined whereas in \cref{alg:lanczos} the coefficients are chosen to enforce orthogonality.

In finite precision arithmetic, we instead have a perturbed recurrence \begin{equation}
    \vec{q}_{i+1} = \frac{1}{\beta_{i}} \left( \vec{A} \vec{q}_i - \alpha_i \vec{q}_i - \beta_{i-1} \vec{q}_{i-1}  - \vec{f}_{i+1} \right),
\end{equation}
where \( \vec{f}_{i+1} \) accounts for local rounding errors made in the computation of \( \vec{q}_{i+1} \).
The involved computations can typically be done stably, so it is reasonable to assume that \( \vec{f}_{i+1}/\beta_i \) has a linear dependence on the machine precision.

While \( \vec{q}_{j} = p_{j}(\vec{A}) \vec{v} \) in exact arithmetic, this is no longer the case in finite precision arithmetic. 
Indeed, the difference between \( \vec{q}_{i+1} \) and \( p_{i+1}(\vec{A})\vec{v} \) depends on the \emph{associated polynomials} of the recurrence applied to the \( \vec{f}_{j} \)'s. 
In particular, 
\begin{equation}
\label{sup:eqn:perturbed_recurrence}
    \vec{q}_{j+1} =
    p_{j+1}(\vec{A}) \vec{v}
    - \sum_{i = 1}^{j+1} \frac{1}{\beta_{i-1}} p_{i,j}(\vec{A})\vec{f}_i,
\end{equation}
where the associated polynomial $p_{i,j+1}(x)$ is defined through the three term recurrence 
\begin{equation}
p_{i,j+1}(x) = \frac{1}{\beta_j} \left( x p_{i,j}(x)  - \alpha_j p_{i,j}(x) -  \beta_{j-1} p_{i,j-1}(x) \right)    
\end{equation} 
with starting condition $p_{i,i}(x) = 1$ and $p_{i,i-1}(x) = 0$; see for instance \cref{sup:sec:associated} or \cite{meurant_06}.

In the case of the Chebyshev polynomials of the first kind, the associated polynomials are the Chebyshev polynomials of the second kind which remain small over $[-1,1]$.
Thus, it can be shown that \( \vec{q}_{i+1} \approx p_{i+1}(\vec{A}) \vec{v} \), where ``\(\approx\)'' roughly means the error has a (mild) polynomial dependence on \( k \) and a linear dependence on the machine precision, along with other reasonable dependencies on the dimension and matrix norm.
As such, the computed modified moments for \( \mu = \mu_{a,b}^T \) can be expected to be near to the true modified moments and bounds based on best polynomial approximation on $[a,b]$ can be expected to hold to close degree as long as \( [\lmin,\lmax]\subset [a,b] \) \cite{clenshaw_55}.

On the other hand, for different \( \{ \alpha_i \}_{i=0}^{\infty} \) and \( \{ \beta_i \}_{i=0}^{\infty} \), for instance those generated by the Lanczos algorithm, the associated polynomials may grow exponentially in \( [\lmin,\lmax]  \) and the modified moments obtained from the finite precision computation may differ greatly from their exact arithmetic counterparts unless very high precision is used.
In fact, this situation includes \( \mu_{a,b}^T \) if \( a \) and \( b \) are not chosen so that \( [\lmin,\lmax] \subset [a,b] \). 
Indeed, the Chebyshev polynomials of the second kind, shifted and scaled to $[a,b]$, grow exponentially with $k$ at points outside of $[a,b]$. 
As such, the error term in \cref{sup:eqn:perturbed_recurrence} can be exponentially large in $k$, resulting in a loss of stability unless the machine precision is set extremely small.

\subsubsection{The Lanczos algorithm}

In the case of \cref{alg:lanczos}, the coefficients are computed adaptively and therefore depend on \( \vec{q}_{i-1} \), \( \vec{q}_i \), and \( \vec{q}_{i+1} \).
It is well known that even if the \( \vec{f}_j \)'s are small, the coefficients produced by \cref{alg:lanczos} run in finite precision arithmetic may differ \emph{greatly} from what would be obtained in exact arithmetic and the Lanczos vectors \( \{ \vec{q}_j \}_{j=1}^{k+1} \) need not be orthogonal. 
Moreover, the tridiagonal matrix \( [\vec{T}]_{:k,:k} \) from the finite precision computation may have multiple nearby eigenvalues even when the eigenvalues would have been well separated in exact arithmetic.
In this sense, the algorithm is unstable, and such instabilities can appear even after only a few iterations.
This has resulted in widespread hesitance to use Lanczos based approaches for spectrum and spectral sum approximation unless the storage and computation heavy full reorthogonalization approach is used
\cite{jaklic_prelovsek_94,silver_roeder_voter_kress_96,aichhorn_daghofer_evertz_vondelinden_03,weisse_wellein_alvermann_fehske_06,ubaru_chen_saad_17,granziol_wan_garipov_19}.

A great deal is known about the Lanczos algorithm in finite precision arithmetic; see for instance \cite{greenbaum_97,meurant_06}.
Notably, \cite{paige_70,paige_76,paige_80} show that orthogonality of the Lanczos vectors is lost precisely when one of the eigenvalues the tridiagonal matrix converges to an eigenvalue of \( \vec{A} \). 
This has lead to the development of schemes which aim to orthogonalize only when necessary \cite{parlett_scott_79}.
Subsequently, \cite{greenbaum_89} shows that the matrix \( [\vec{T}]_{:k,:k} \) obtained by an implementation of the Lanczos algorithm run in finite precision can be viewed as the output the Lanczos algorithm run in exact arithmetic on a certain ``nearby'' problem. 
The bounds of \cite{greenbaum_89} are described at a high level in the context of spectrum approximation in  \cite[Appendix D]{chen_trogdon_ubaru_21}, and they provide insight as to why Lanczos based quadrature may still work in finite precision arithmetic.
However, that the bounds of \cite{greenbaum_89} have dependencies on the machine precision which are worse than linear and seemingly pessimistic in practice. 

While the tridiagonal matrix produced by Lanczos in finite precision arithmetic can be very different from what would be obtained in exact arithmetic, the quadrature approximation derived from the eigenvalues and squares of the first components of the eigenvectors of \( [\vec{T}]_{:k,:k} \) still computes modified moments accurately.
Specifically, \cite[Theorem 1]{knizhnerman_96} shows that as long as \( k \) is not too large, for any \( j \leq 2k-1 \), 
\begin{equation}
    \vec{v}^\cT \tilde T_j(\vec{A}) \vec{v} \approx (\vec{e}_0)^\cT \tilde T_j([\vec{T}]_{:k,:k}) \vec{e}_0,
\end{equation}
where \( \tilde{T}_j \) is the Chebyshev polynomial \( T_j \) scaled to an interval slightly larger than \( [\lmin,\lmax] \). 
Thus,  best polynomial approximation on $[a,b]$ hold to close approximation if the best approximation among all polynomials is replaced with the best approximation among polynomials with Chebyshev coefficients which are not too large.
Moreover, this essentially implies that computing the modified moments for the Chebyshev distribution can be done stably using the information generated by Lanczos.

A full analysis of the finite precision behavior of the algorithms studied in this paper is beyond the scope of this paper.
Instead, we leave readers with the following heuristics which are supported by the theory mentioned above and by the numerical examples below:
In finite precision arithmetic, Lanczos based Gaussian quadrature will perform nearly as well as, or quite often significantly better than, properly scaled Chebyshev based approaches.
However, while the Chebyshev based approaches require knowledge of \( a \) and \( b \) with \( [\lmin,\lmax] \subset [a,b] \) to avoid exponential instabilities, Lanczos based approaches do not require a priori knowledge of the spectrum of \( \vec{A} \).
Further discussion can be found in \cite{chen_24}.

\section{Additional Proofs}

\subsection{Proof of \cref{sup:eqn:perturbed_recurrence}}
\label{sup:sec:associated}

For any $i\geq 0$, define the associated polynomials by the recurrence
\begin{equation}
    \label{eqn:assoc_three_term}
    p_{i,j+1}(x) = \frac{1}{\beta_j} \left( x p_{i,j}(x)  - \alpha_j p_{i,j}(x) -  \beta_{j-1} p_{i,j-1}(x) \right)
    ,\qquad j =i,i+1,\ldots
\end{equation}
with starting condition $p_{i,i}(x) = 1$ and $p_{i,i-1}(x) = 0$.

\begin{lemma}
        Suppose 
    \begin{equation}
        \vec{q}_{j+1} = \frac{1}{\beta_{j}} \left( \vec{A} \vec{q}_j - \alpha_i \vec{q}_j - \beta_{j-1} \vec{q}_{j-1}  - \vec{f}_{j+1} \right) ,
    \end{equation}
    where $\vec{q}_0 = \vec{v}$ and $\beta_0 = \beta_{-1} = 0$.
    Then,
    \begin{align*}
        \vec{q}_{j+1} =
        p_{j+1}(\vec{A}) \vec{v}
        - \sum_{i = 1}^{j+1} \frac{1}{\beta_{i-1}} p_{i,j}(\vec{A})\vec{f}_i
    \end{align*}
\end{lemma}

\iftrue
\begin{proof}
    Define \( \vec{\Delta}_j := \vec{q}_j - p_j(\vec{A})\vec{v} \).
    By definition we have \( \vec{\Delta}_0 = \vec{0} \) and it is easy to verify that
     \begin{equation}
        \vec{q}_1 
        = \frac{1}{\beta_0} \left(\vec{A} \vec{q}_0 - \alpha_0 \vec{q}_0 - \vec{f}_1\right)
        = p_{1}(\vec{A}) \vec{v} -  \frac{1}{\beta_0}p_{1,1}(\vec{A}) \vec{f}_1.
    \end{equation}
    Thus, \( \vec{\Delta}_1 = -\vec{f}_1 / \beta_0 \). 
    
    For \( j > 0 \),
    \begin{align*}
        \vec{q}_{j+1} 
        &= \frac{1}{\beta_{j}} \left( \vec{A} \vec{q}_{j} - \alpha_{j} \vec{q}_{j} - \beta_{j-1} \vec{q}_{j-1} - \vec{f}_{j+1} \right) 
        \\&= \frac{1}{\beta_{j}} \left( \vec{A} p_{j}(\vec{A})\vec{v} - \alpha_{j} p_{j}(\vec{A})\vec{v} - \beta_{j-1} p_{j-1}(\vec{A})\vec{v} \right) 
        \\&\hspace{5em} + \frac{1}{\beta_{j}} \left( \vec{A} \vec{\Delta}_{j} - \alpha_{j} \vec{\Delta}_{j} - \beta_{j-1} \vec{\Delta}_{j-1} \right) - \frac{1}{\beta_{j}} \vec{f}_{j+1}
        \\&= p_{j+1}(\vec{A})\vec{v} + \frac{1}{\beta_{j}} \left( \vec{A} \vec{\Delta}_{j} - \alpha_{j} \vec{\Delta}_{j} - \beta_{j-1} \vec{\Delta}_{j-1} \right) - \frac{1}{\beta_{j}} \vec{f}_{j+1}.
    \end{align*}
    Thus, we see that
    \begin{equation}
        \vec{\Delta}_{j+1} =  \frac{1}{\beta_{j}} \left( \vec{A} \vec{\Delta}_{j} - \alpha_{j} \vec{\Delta}_{j} - \beta_{j-1} \vec{\Delta}_{j-1} \right) - \frac{1}{\beta_{j}}\vec{f}_{j+1}.
    \end{equation}

    By induction, assume that
    \begin{equation}
        \vec{\Delta}_{j} = - \sum_{\ell=1}^{j} \frac{1}{\beta_{\ell-1}} p_{\ell,j}(\vec{A}) \vec{f}_\ell
        ,\qquad
        \vec{\Delta}_{j-1} = - \sum_{\ell=1}^{j-1} \frac{1}{\beta_{\ell-1}} p_{\ell,j-1}(\vec{A}) \vec{f}_\ell.
    \end{equation}
    Therefore, using \cref{eqn:assoc_three_term},
    \begin{align*}
        \vec{\Delta}_{j+1} 
        &= -  \sum_{\ell=1}^{j-1} \frac{1}{\beta_j}\left( \frac{1}{\beta_{\ell-1}} \vec{A} p_{\ell,j}(\vec{A}) \vec{f}_\ell  - \frac{1}{\beta_{\ell-1}} \alpha_j p_{\ell,j}(\vec{A}) \vec{f}_\ell - \frac{\beta_{j-1}}{\beta_{\ell-1}} p_{\ell,j-1}(\vec{A})\vec{f}_\ell \right)
        \\&\hspace{5em}-\frac{1}{\beta_j} \left( \frac{1}{\beta_{j-1}} \vec{A} p_{j,j}(\vec{A}) \vec{f}_j  - \frac{1}{\beta_{j-1}} \alpha_{j} p_{j,j}(\vec{A}) \vec{f}_j  \right)
        - \frac{1}{\beta_j} \vec{f}_{j+1}
        \\&= -  \sum_{\ell=1}^{j-1} \frac{1}{\beta_{\ell-1}} p_{\ell,j+1}(\vec{A}) \vec{f}_\ell
        - \frac{1}{\beta_{j-1}}p_{j,j+1}(\vec{A})
        - \frac{1}{\beta_j} p_{j+1,j+1}(\vec{A}) \vec{f}_{j+1}
        \\&= -  \sum_{\ell=1}^{j+1} \frac{1}{\beta_{\ell-1}} p_{\ell,j+1}(\vec{A}) \vec{f}_\ell.
    \end{align*}
\end{proof}

\subsection{Proof of \cref{thm:smoothing_wass}}
\label{sec:smoothing_wass}

\begin{proof}[Proof of \cref{thm:smoothing_wass}]
The first statement follows by noting that
\allowdisplaybreaks
    \begin{align*}
        \W(\Upsilon,\Upsilon_\sigma)
        &= \int_{-\infty}^{\infty} \left| \int_{-\infty}^{\infty} \bOne[x-t<0] \d\Upsilon(t) - \int_{-\infty}^{\infty} G_\sigma(x-t) \d\Upsilon(t) \right| \d{x}
        \\&= \int_{-\infty}^{\infty} \left| \int_{-\infty}^{\infty} ( \bOne[x-t<0] G_\sigma(x-t) ) \d\Upsilon(t) \right| \d{x}
        \\&\leq \int_{-\infty}^{\infty} \left( \int_{-\infty}^{\infty} \left|  \bOne[x-t<0]-  G_\sigma(x-t) \right|  \left| \d\Upsilon(t) \right| \right) \d{x}
        \\&= \int_{-\infty}^{\infty} \left( \int_{-\infty}^{\infty} \left|  \bOne[x-t<0]-  G_\sigma(x-t) \right| \d{x} \right) \left| \d\Upsilon(t) \right|
        \\&= \int_{-\infty}^{\infty} \W(\bOne[\cdot < 0], G_\sigma)  \left| \d\Upsilon(t) \right|
        \\&= \W(\bOne[\:\cdot < 0], G_\sigma) .
    \end{align*}

    Next, we split the integral and exchange the bounds of integration to find
    \begin{align*}
        \int_{-\infty}^{\infty} \left|  \bOne[x<0]-  G_\sigma(x) \right| \d{x}
        &= \int_{0}^{\infty} ( 1-G_\sigma(x) ) \d{x} + \int_{-\infty}^{0} G_\sigma(x) \d{x}
        \\&= \int_{0}^{\infty} \int_{x}^\infty \d G_\sigma(t) \d{t} + \int_{-\infty}^{0} \int_{-\infty}^{x} \d G_\sigma(t) \d{x}
        \\&= \int_{0}^{\infty} \int_{0}^t \d{x} \d G_\sigma(t) + \int_{-\infty}^{0} \int_{t}^{0} \d{x} \d G_\sigma(t)
        \\&= \int_{0}^{\infty} t \d G_\sigma(t) + \int_{-\infty}^{0} -t \d G_\sigma(t)
        \\&= \int_{-\infty}^{\infty} |t| \d G_\sigma(t).
    \end{align*}
    Finally, by the Cauchy--Schwarz inequality,
    \begin{equation*}
        \int_{-\infty}^{\infty} |t| \d  G_\sigma(t)
        \leq \left( \int_{-\infty}^{\infty} |t|^2 \d  G_\sigma(t) \right)^{1/2} = \sigma.
    \end{equation*}
\end{proof}

\section{Numerical experiments and applications}
\label{sup:sec:numerical_experiments}

In this section, we provide additional numerical experiments showing the finite precision behavior of the algorithm studied.

\subsection{Finite precision convergence of SLQ}
\begin{figure}
    \includegraphics[width=\textwidth]{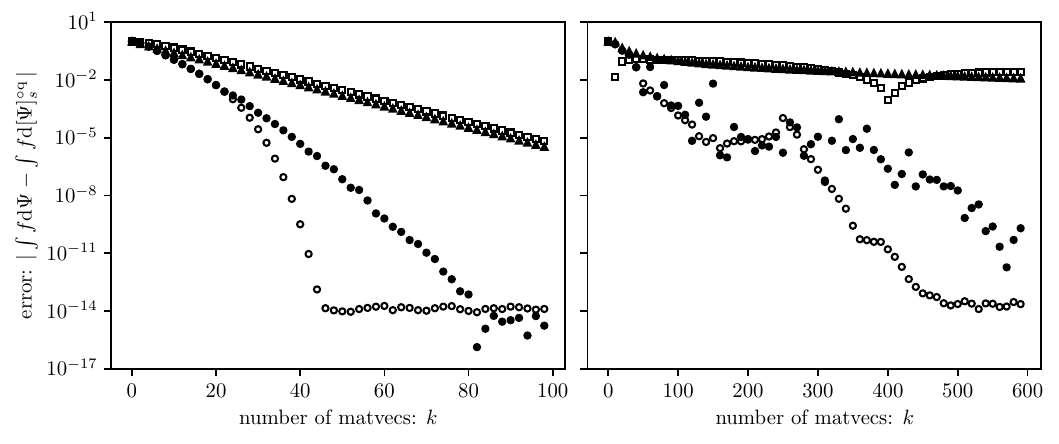}
    \caption{
    Errors for approximating \( \tr(f(\vec{A})) \) on various problems in finite precision arithmetic.
    \emph{Legend}:
    Gaussian quadrature with reorthogonalization
    ({\protect\raisebox{0mm}{\protect\includegraphics[scale=.7]{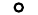}}}) and without reorthogonalization
    ({\protect\raisebox{0mm}{\protect\includegraphics[scale=.7]{imgs/legend/circle.pdf}}}), 
    approximate quadrature by approximation 
    ({\protect\raisebox{0mm}{\protect\includegraphics[scale=.7]{imgs/legend/tri.pdf}}}),
    and quadrature by interpolation 
    ({\protect\raisebox{0mm}{\protect\includegraphics[scale=.7]{imgs/legend/square_empty.pdf}}}).
    \emph{Takeaway}: Without reorthogonalization the convergence of Gaussian quadrature is slowed. 
    However, the method still converges and can even outperform the other methods.
    }
    \label{fig:FP_effects}
\end{figure}

In this example, we consider several experiments where orthogonality is lost and the effects of finite precision arithmetic are easily observed.
In both experiments we use diagonal matrices scaled so \( \| \vec{A} \|_2 = 1 \) and set \( \vec{v} \) to have uniform entries.
We set \( a,b \) as the largest and smallest eigenvalues respectively and again use \( \mu = \mu_{a,b}^T \) for the interpolatory and quadrature by approximations.

In the first experiment, shown in the left panel of \cref{fig:FP_effects}, the eigenvalues of \( \vec{A} \) are distributed according to the model problem \cite{strakos_91,strakos_greenbaum_92} 
and \( f(x) = 1/x \).
Specifically, the eigenvalues are given by
\begin{equation}
    \lambda_1 = 1
    ,\qquad \lambda_n = \kappa
    ,\qquad \lambda_i = \lambda_1 + \left( \frac{i-1}{n-1} \right) \cdot (\kappa -1) \cdot \rho^{n-i}
    ,\qquad i=2,\ldots,n-1
\end{equation}
with selected parameters \( n = 300 \), \( \kappa = 10^3 \), and \( \rho = 0.85 \).
The model problem is a standard problem in the analysis of the finite precision behavior of Lanczos based algorithms, especially in the context of solving linear systems of equations.
This is because the exponential spacing of the eigenvalues is favorable to Lanczos based linear system solvers in exact arithmetic yet simultaneously causes the Lanczos algorithm to rapidly lose orthogonality in finite precision arithmetic.
In the second experiment, shown in the right panel \cref{fig:FP_effects}, we use the \( n=9664 \) eigenvalues of the California matrix from the sparse matrix suite \cite{davis_hu_11} and the function \( f(x) = |x| \).

In both cases, the Jacobi matrices produced by Lanczos, with or without reorthogonalization, differ greatly; i.e. the difference of the matrices is on the order of \( \| \vec{A} \|_2 \). 
Even so, the modified moments for \( \mu = \mu_{a,b}^T \) obtained by \cref{alg:moments_cheb,alg:moments_lanc} differ only in the 12th digit and 14th digits on the two matrices respectively.
Using one approach in place of the other does not noticeably impact the convergence of the quadrature by interpolation approximations.

\printbibliography

%% file: SISC_shared.tex
\usepackage{lipsum}
\usepackage{amsfonts}
\usepackage{graphicx}
\usepackage{epstopdf}

\ifpdf
  \DeclareGraphicsExtensions{.eps,.pdf,.png,.jpg}
\else
  \DeclareGraphicsExtensions{.eps}
\fi

\newsiamremark{remark}{Remark}
\newsiamremark{hypothesis}{Hypothesis}
\crefname{hypothesis}{Hypothesis}{Hypotheses}
\newsiamthm{assumption}{Assumption}

\headers{Randomized matrix-free quadrature}{T. Chen, T. Trogdon, and S. Ubaru}

\title{Randomized matrix-free quadrature: unified and uniform bounds for stochastic Lanczos quadrature and the kernel polynomial method \thanks{
\funding{This material is based on work supported by the National Science Foundation under Grant Nos. DGE-1762114 (TC), DMS-1945652 (TT).
Any opinions, findings, and conclusions or recommendations expressed in this material are those of the authors and do not necessarily reflect the views of the National Science Foundation.}}}

\author{Tyler Chen\thanks{New York University 
  (\email{tyler.chen@nyu.edu}).}
\and Thomas Trogdon\thanks{University of Washington
  (\email{trogdon@uw.edu}).}
\and Shashanka Ubaru\thanks{IBM Watson (\email{Shashanka.ubaru@ibm.com}).}}

\usepackage{amsopn}

\ifpdf
\hypersetup{
  pdftitle={Randomized matrix-free quadrature for spectrum and spectral sum approximation},
  pdfauthor={T. Chen, T. Trogdon, and S. Ubaru}
}
\fi

\usepackage[backend=biber,style=alphabetic,firstinits=true,maxnames=99,doi=false,eprint=false,url=false]{biblatex}
\usepackage[T1]{fontenc}
\usepackage[utf8]{inputenc}
\usepackage[margin = 1.35in, top = 1.25in,bottom = 1.25in]{geometry}
\DeclareSymbolFont{bbold}{U}{bbold}{m}{n}
\DeclareSymbolFontAlphabet{\mathbbold}{bbold}

\usepackage{thmtools}
\usepackage{thm-restate}

\usepackage{mleftright}
\mleftright

\usepackage{comment}

\usepackage{algorithm}
\usepackage{algorithmicx}
\usepackage[noend]{algpseudocode}

\newenvironment{labelalgorithm}[4][t]{%
\begin{algorithm}[#1]
\newcommand{\thealgorithmname}{#3}
\customlabel{alg:name:#2}{\textproc{#3}}
\caption{#4}\label{alg:#2}
}{\end{algorithm}}

\numberwithin{algorithm}{section}

\makeatletter
\newcommand{\customlabel}[2]{%
   \protected@write \@auxout {}{\string \newlabel {#1}{{#2}{\thepage}{#2}{#1}{}} }%
   \hypertarget{#1}{}%
}
\makeatother

\usepackage{listings}
\usepackage{tikz}

\definecolor{Base02}{HTML}{1e3668}

\definecolor{Red}{HTML}{dc322f}
\definecolor{Magenta}{HTML}{d33682}
\definecolor{Violet}{HTML}{6c71c4}
\definecolor{Cyan}{HTML}{2aa198}
\definecolor{Green}{HTML}{859900}

\usepackage{hyperref}
\hypersetup{
    colorlinks,
    linkcolor=Base02,
    citecolor=Base02,
    urlcolor=Base02
}

\usepackage{marginnote}

\usepackage{enumitem}

\usepackage{graphicx}
\usepackage{subcaption}
\usepackage{booktabs}
\usepackage{array}
\usepackage{multirow}

\usepackage{afterpage}
\usepackage{pdflscape}

\usepackage{setspace}

\usepackage{mathtools}

\usepackage{bm}

\newcommand{\bOne}{\ensuremath{\mathbbold{1}}}
\newcommand{\EE}{\ensuremath{\mathbb{E}}}

\newcommand{\PP}{\ensuremath{\mathbb{P}}}

\newcommand{\lmin}{\ensuremath{\lambda_{\textup{min}}}}
\newcommand{\lmax}{\ensuremath{\lambda_{\textup{max}}}}

\newcommand{\samp}[2][]{#1\langle #2 #1\rangle}

\usepackage{etoolbox}

\newcommand{\qq}[3][]{%
    \ifstrempty{#1}{%
        [#2]_{#3}^{\circ\textup{q}}
    }{%
        [#2]_{#3}^{\textup{{#1}q}}
    }%
}

\newcommand{\ff}[3][]{%
    \ifstrempty{#1}{%
        [#2]_{#3}^{\circ\textup{p}}
    }{%
        [#2]_{#3}^{\textup{{#1}p}}
    }%
}

\newcommand{\R}{\mathbb{R}}

\renewcommand{\d}{\ensuremath{\,\mathrm{d}}}
\newcommand{\im}{\bm{i}}

\newcommand{\W}{d_{\mathrm{W}}}

\renewcommand{\vec}{\mathbf}

\newcommand{\F}{\textup{\textsf{F}}}
\newcommand{\cT}{*}
\newcommand{\tr}{\operatorname{tr}}

\usepackage{xspace}

\newcommand{\nv}{m}

\newcommand{\fA}{f(\vec{A})}

\usepackage[normalem]{ulem}
\newcommand{\stkout}[1]{\ifmmode\text{\sout{\ensuremath{#1}}}\else\sout{#1}\fi}